\documentclass[11pt, reqno, oneside]{article}
\usepackage{geometry} 
\usepackage[affil-it]{authblk}

\usepackage{amsmath, amsthm, amsfonts, amssymb, amsbsy, bigstrut, graphicx, enumerate,  upref, longtable, comment, booktabs, array, caption, subcaption}

\usepackage[breaklinks]{hyperref}

\usepackage[numbers, sort&compress]{natbib}

\newcommand{\vol}{\mathrm{Vol}}
\newcommand{\I}{\mathrm{i}}

\newcommand{\me}{\mathcal{E}}

\newcommand{\ms}{\mathcal{S}}

\newtheorem{thm}{Theorem}[section]
\newtheorem{lmm}[thm]{Lemma}
\newtheorem{cor}[thm]{Corollary}

\newtheorem{defn}[thm]{Definition}

\theoremstyle{definition}

\newcommand{\cov}{\mathrm{Cov}}

\newcommand{\ee}{\mathbb{E}}

\newcommand{\ma}{\mathcal{A}}

\newcommand{\mf}{\mathcal{F}}

\newcommand{\pp}{\mathbb{P}}

\newcommand{\tr}{\operatorname{Tr}}

\newcommand{\var}{\mathrm{Var}}
\newcommand{\ve}{\varepsilon}

\numberwithin{equation}{section}

\renewcommand{\Re}{\operatorname{Re}}

\usepackage[utf8]{inputenc}
\usepackage{amsmath}
\usepackage{amsfonts}
\usepackage{bbm}
\usepackage{mathtools}
\usepackage{tikz}
\usetikzlibrary{decorations.pathreplacing}
\usetikzlibrary {arrows.meta}
\usepackage{amsthm}
\usepackage[english]{babel}
\usepackage{fancyhdr}
\usepackage{amssymb}
\usepackage{enumitem}
\usepackage{qtree}
\usepackage{mathrsfs}

\newcommand{\Z}{\mathbb{Z}}

\newcommand{\R}{\mathbb{R}}

\newcommand{\C}{\mathbb{C}}

\newcommand{\E}{\mathbb{E}}

\renewcommand{\bar}{\overline}
\renewcommand{\P}{\mathbb{P}}

\renewcommand{\tilde}{\widetilde}

\newcommand{\sun}{\mathrm{SU}(N)}
\newcommand{\sutwo}{\mathrm{SU}(2)}

\newcommand{\uone}{\mathrm{U}(1)}

\begin{document}
\title{A scaling limit of $\mathrm{SU}(2)$ lattice Yang--Mills--Higgs theory}
\author{Sourav Chatterjee\thanks{Department of Statistics, Stanford University, USA. Email: \href{mailto:souravc@stanford.edu}{\tt souravc@stanford.edu}. 
}}
\affil{Stanford University}

\maketitle

\begin{abstract}
The construction of non-Abelian Euclidean Yang--Mills theories in dimension four, as scaling limits of lattice Yang--Mills theories or otherwise, is a central open question of mathematical physics. This paper takes the following small step towards this goal. 
In any dimension $d\ge 2$, we construct a scaling limit of $\mathrm{SU}(2)$ lattice Yang--Mills theory coupled to a Higgs field (under the degenerate potential) transforming in the fundamental representation of $\mathrm{SU}(2)$. After unitary gauge fixing and taking the lattice spacing $\varepsilon\to 0$, and simultaneously taking  the gauge coupling constant $g\to 0$ and the Higgs length $\alpha\to \infty$ in such a manner that $\alpha g$ is always equal to $c\varepsilon$ for some fixed $c$ and $g= O(\varepsilon^{50d})$, a stereographic projection of the gauge field is shown to converge to a massive Gaussian field. This gives the first construction of a scaling limit of a non-Abelian lattice Yang--Mills theory in a dimension higher than two, as well as the first rigorous proof of mass generation by the Higgs mechanism in such a theory. Analogous results are proved for $\uone$ theory as well. The question of constructing a non-Gaussian scaling limit remains open.
\end{abstract}

\tableofcontents

\section{Lattice Yang--Mills theory with a Higgs field}\label{intro}
Quantum Yang--Mills theories are the building blocks of the Standard Model of particle physics. While these are quantum theories in Minkowski space, there are analogues in Euclidean space,  known as  Euclidean Yang--Mills theories. Physicists believe that calculations in the Euclidean theories can be carried over to the quantum theories. The constructive QFT program aims to justify this mathematically, although it is not yet well-developed for Yang--Mills theories~\cite{glimmjaffe87, jaffewitten06}. 

Euclidean Yang--Mills theories are supposed to be classical statistical mechanical models of quantum Yang--Mills theories. However, we do not know how to construct even these classical models in a rigorous way in dimensions higher than two.  Lattice Yang--Mills theories, also known as lattice gauge theories, are discretized versions of Euclidean Yang--Mills theories that are mathematically well-defined. 

\subsection{Definition}\label{lgtsec}
Let $G$ be a compact Lie group, called the {\it gauge group} of the theory. Let $d\ge 2$ and $L\ge 1$ be two integers, and let $\Lambda := \{-L,\ldots,L\}^d$. Our theory will be defined on $\Lambda$, and eventually $L$ will be taken to infinity. Throughout this paper, we will consider theories on $\Lambda$ with {\it periodic boundary condition}, meaning that opposite faces of $\Lambda$ are identified to be the same, giving $\Lambda$ the structure of a torus. 

Let $E$ denote the set of edges of $\Lambda$ that are oriented in the  positive direction (i.e., edges like $e= (x,y)$ where $x\prec y$ in the lexicographic ordering). Note that this definition makes sense even after identifying opposite faces, because two edges that are identified to be the same are either both positively oriented, or both negatively oriented. The space of gauge field configurations of the theory is $\Sigma := G^E$. That is, a gauge field configuration is an assignment of group elements to edges. Let us denote configurations as $U = (U_e)_{e\in E}$. If $e=(x,y)\in E$ and $e' = (y,x)$ is the negatively oriented version of $e$, then we define $U_{e'} := U_e^{-1}$.

A {\it plaquette} in $\Lambda$ is a square bounded by four edges. Given a configuration $U$ and a plaquette $p$ bounded by four directed edges $e_1,e_2,e_3,e_4$ joined end-to-end (see Figure~\ref{plaquette1}), we define
\[
U_p := U_{e_1}U_{e_2}U_{e_3}U_{e_4}.
\]
Note that there is an ambiguity here about the choice of the first edge and the direction of traversal, but that would be immaterial because we will only deal with the quantities that are not affected by these choices.

\begin{figure}
\begin{center}
\begin{tikzpicture}[scale = 2.5]
\draw [-Stealth, thick] (0,0) to (1,0);
\draw [-Stealth, thick] (1,0) to (1,1);
\draw [-Stealth, thick] (1,1) to (0,1);
\draw [-Stealth, thick] (0,1) to (0,0);
\node at (.5,-.2) {$e_1$};
\node at (1.2,.5) {$e_2$};
\node at (.5,1.2) {$e_3$};
\node at (-.2,.5) {$e_4$};
\end{tikzpicture}
\caption{A plaquette bounded by four directed edges joined end-to-end.\label{plaquette1}}
\end{center}
\end{figure}
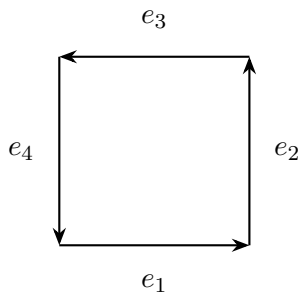

Let $K$ be a positive integer. A {\it Higgs field} $\phi$ is a map from $\Lambda$ into $\C^K$. Let $\Gamma := (\C^K)^\Lambda$ denote the set of all Higgs fields (where again, vertices on opposite faces are identified to be the same). The space of configurations for our lattice gauge theory coupled to a Higgs field is $\Sigma \times \Gamma$.

Choose a number $g>0$, called the {\it gauge coupling constant}, and a number $\alpha >0$, called the {\it Higgs length}.  Let $\rho$ be a finite dimensional representation of $G$, and $\chi_\rho$ be its character. 
Let $\sigma$ be a representation of $G$ in $\C^K$, which is unitary with respect to the usual inner product on $\C^K$. Let $W:[0,\infty) \to \R$ be a function, called the {\it Higgs potential,} that grows faster than quadratic at infinity.  The {\it Yang--Mills action} is defined as
\begin{align*}
S(U, \phi) &:=\frac{1}{g^2}\sum_{p\in P} \Re(\chi_\rho(U_p)) + \alpha^2 \sum_{e=(x,y)\in E} \Re(\phi_x^*\sigma(U_e)\phi_y)  - \sum_{x\in \Lambda} W(\|\phi_x\|).
\end{align*}
The theory defines a probability measure $\mu$ on $\Sigma \times \Gamma$ as 
\[
d\mu(U,\phi) = Z^{-1}e^{S(U, \phi)} \prod_{e\in E} dU_e \prod_{x\in \Lambda} d\phi_x,
\]
where $Z$ is the normalizing constant, $dU_e$ denotes normalized Haar measure on $G$ and $d\phi_x$ denotes Lebesgue measure on $\C^K$. 

Our main results will be for an {\it infinite volume Gibbs measure} for this model on the lattice $\Z^d$, when $W$ is the degenerate potential that is zero at $1$ and infinity away from~$1$ (which simply means that $\phi_x$ is forced to take values on the unit sphere in $\C^K$). Such Gibbs measures are obtained by taking weak limits of the probability measure $\mu$ defined above as $L\to \infty$. Since we are considering only periodic boundary condition, any Gibbs measure obtained like this is automatically translation invariant. Until we arrive at the main results, however, we will continue the discussion around the models in finite volume (that is, on $\Lambda$).

\subsection{Gauge transforms and gauge fixing}
Let $\Theta := G^\Lambda$. That is, an element $\theta\in \Theta$ assigns group elements to vertices. Elements of $\Theta$ are called {\it gauge transforms.} Gauge transforms form a group under pointwise multiplication, which acts on $\Sigma \times \Gamma$ as follows. For $(U,\phi)\in \Sigma \times \Gamma$ and $\theta \in \Theta$, we define $(V, \psi):= \theta(U,\phi)$ as
\[
V_e = \theta_x U_e \theta_y^{-1}, \ \ \psi_x = \sigma(\theta_x) \phi_x,
\]
where $e= (x,y)$. Using the unitarity of $\sigma$, it is not hard to verify that $S$ is {\it gauge invariant}, that is, $S(\theta(U,\phi)) = S(U,\phi)$. Consequently, the theory has {\it gauge symmetry,} meaning that the probability measure $\mu$ is preserved by gauge transforms. 

Any procedure that selects exactly one element from each orbit of $\Theta$ is called a method for {\it fixing a gauge}. To understand the behavior of a lattice Yang--Mills theory, it suffices to understand a gauge fixed version of the theory, because physically relevant observables are required to be invariant under gauge transforms.

\subsection{Example: $\mathrm{U}(1)$ lattice Yang--Mills theory with a Higgs field}\label{uonesec}
Let $G:=\uone$, $K=1$, and $\rho = \sigma =  $ the fundamental representation of $\uone$ (i.e., the representation that acts on $\C$ by ordinary multiplication). Let $W$ be the degenerate Higgs potential that is $0$ on the unit circle $S^1$ and $\infty$ outside. This potential is a common choice in the literature~\cite{fradkinshenker79, seiler82}, partly because it simplifies calculations without damaging the physical relevance of the theory. This forces the Higgs field at each site to be an element of $S^1$. We may therefore take the space of configurations to be $\Sigma \times \Gamma$ where $\Sigma = \uone^E$ and $\Gamma = (S^1)^\Lambda $. The Yang--Mills action on this space is 
\begin{align}\label{action1}
S(U,\phi) = \frac{1}{g^2}\sum_{p\in P} \Re(U_p) + \alpha^2 \sum_{e=(x,y)\in E} \Re(\phi_x^*U_e \phi_y),
\end{align}
where $\phi_x^*$ is the conjugate of the complex number $\phi_x$. 
The probability measure $\mu$ on $\Sigma \times \Gamma$ is now defined as 
\[
d\mu(U,\phi) = Z^{-1}e^{S(U, \phi)} \prod_{e\in E} dU_e \prod_{x\in \Lambda} d\phi_x,
\]
where $Z$ is the normalizing constant, $dU_e$ denotes normalized Haar measure on $G$ and $d\phi_x$ denotes the uniform probability measure on $S^1$ (instead of Lebesgue measure on $\C$). 

A standard way to fix the gauge in the above theory, known as {\it unitary gauge fixing}, is defined as follows. Take any configuration $(U, \phi)$. For each $x\in \Lambda$, let $\theta_x := \phi_x^*$. Let $(V,\psi) := \theta(U, \phi)$. Note that for any edge $e=(x,y)\in E$, $V_e = \phi_x^*U_e \phi_y$, and for any vertex $x\in \Lambda$, $\psi_x = 1$. Thus, if our interest lies in understanding only the behavior of gauge invariant observables, it suffices to study the field $V$.

\subsection{Example: $\sutwo$ lattice Yang--Mills theory with a Higgs field}\label{sutwosec}
Let $G:=\sutwo$, $K=2$, and $\rho = \sigma =  $ the fundamental representation of $\sutwo$ (i.e., the representation that acts on $\C^2$ in the usual way). Let $W$ be the degenerate Higgs potential that is $0$ on the unit sphere $S^3$ of $\R^4$ (i.e., the set $\{(w,z)\in \C^2: |w|^2 + |z|^2 = 1\}$, which is the same as the unit sphere $S^3 \subseteq \R^4$) and $\infty$ outside. This forces the Higgs field at each site to be an element of $S^3$. Thus,  the space of configurations is $\Sigma \times \Gamma$ where $\Sigma = \sutwo^E$ and $\Gamma = (S^3)^\Lambda $. The Yang--Mills action is 
\begin{align}\label{action2}
S(U,\phi) = \frac{1}{g^2}\sum_{p\in P} \Re(\tr(U_p)) + \alpha^2 \sum_{e=(x,y)\in E} \Re(\phi_x^*U_e \phi_y),
\end{align}
where $\phi_x^*$ is the conjugate transpose of the vector $\phi_x$.  The probability measure $\mu$ on $\Sigma \times \Gamma$ is now 
\[
d\mu(U,\phi) = Z^{-1}e^{S(U, \phi)} \prod_{e\in E} dU_e \prod_{x\in \Lambda} d\phi_x,
\]
where $Z$ is the normalizing constant, $dU_e$ denotes normalized Haar measure on $G$ and $d\phi_x$ denotes the uniform probability measure on $S^3$. 

Unitary gauge fixing in this setting is defined as follows. Take any configuration $(U, \phi)$. For each $x\in \Lambda$, it is not hard to see that there is a unique element $\theta_x\in \sutwo$ such that 
\[
\theta_x \phi_x = e_1 := 
\begin{pmatrix}
1 \\
0
\end{pmatrix}.
\]
Let $(V,\psi) := \theta(U, \phi)$. Then for any edge $e=(x,y)\in E$, $V_e = \theta_xU_e \theta_y^*$, and for any vertex $x\in \Lambda$, $\psi_x = e_1$. Thus, as before, if our interest lies in understanding only the behavior of gauge invariant observables, it suffices to study the field $V$. 

\section{Random distributional $1$-forms}
The main results of this paper are in the next section. The results show that the $\uone$ and the $\sutwo$ lattice Yang--Mills theories defined in Subsections \ref{uonesec} and \ref{sutwosec} converge, under appropriate scalings  of the lattice, to {\it random distributional $1$-forms}. In this section, we define these objects and describe some of their essential properties.

\subsection{Definition}
Recall that a function $f:\R^d\to \R$ is said to be {\it rapidly decaying} if for any $\alpha >0$, there is some $C$ such that $|f(x)|\le C(1+\|x\|)^{-\alpha}$ for all $x\in \R^d$. Let $\ms(\R^d)$ denote the space of real-valued Schwartz functions on $\R^d$, i.e., infinitely differentiable functions $f$ such that $f$ and all its derivatives of all orders are rapidly decaying functions.  We define a {\it Schwartzian $1$-form} on $\R^d$ as a $d$-tuple $f= (f_1,\ldots,f_d)$ of real-valued Schwartz functions on $\R^d$. As a matter of notation, we will write $f= f_1dx_1+\cdots +f_d dx_d$. Let $\ma(\R^d)$ denote the space of Schwartzian $1$-forms on $\R^d$.

Let $(\Omega, \mf ,\P)$ be a probability space. We define a {\it random $1$-form} on $\R^d$ to be a measurable map $X:\Omega \times \R^d \to \R^d$ that satisfies, for any compact set $K\subseteq \R^d$, any $1\le i\le d$, and almost every $\omega \in \Omega$,
\[
\int_K |X_i(\omega,x)| dx < \infty.
\]
Note that we are not imposing any smoothness assumptions on $X$. A random $1$-form $X$ defines a random linear functional on $\ma(\R^d)$, which we also denote by $X$, as follows:
\[
X(f)(\omega) := \int_{\R^d} \sum_{i=1}^dX_i(\omega, x) f_i(x) dx.
\]
Generalizing this type of random linear functional, we will say that any collection of random variables $X = (X(f))_{f\in \ma(\R^d)}$ is a {\it random distributional $1$-form} if it satisfies the condition that for any $f,g\in \ma(\R^d)$ and any $a,b\in \R$, 
\[
X(a f+ bg) = aX(f) + bX(g) \ \ \text{with probability one.}
\]
We will say that a sequence of random distributional $1$-forms $\{X_n\}_{n\ge 1}$ converges in law to a random distributional $1$-form $X$ if for any $f\in \ma(\R^d)$, the random variable $X_n(f)$ converges in law to $X(f)$. By linearity of random distributional $1$-forms, it is not hard to see that if $X_n$ converges in law to $X$, then for any $k$ and any $f_1,\ldots, f_k$, the random vector $(X_n(f_1),\ldots,X_n(f_k))$ converges in law to $(X(f_1),\ldots,X(f_k))$.

\subsection{The Euclidean Proca field}
We now define a particular random distributional $1$-form, called the {\it Euclidean Proca field}. This field was defined by \citet{gross74} and \citet{ginibrevelo75} as the Euclidean version of the Proca field from quantum field theory in Minkowski spacetime~\cite{stuckelberg38}. As we will see below, this field is massive (i.e., has exponential decay of correlations). The massless version of the Euclidean Proca field arises in lattice Maxwell theory~\cite[Chapter 22]{glimmjaffe87}. 

Recall that $d\ge 2$. Given $\lambda >0$, define a function $K_\lambda: \R^d \setminus \{0\} \to \R$ as
\begin{align}\label{klambdadef}
K_\lambda(x) := \int_0^\infty \frac{1}{(4\pi t)^{d/2}} \exp\biggl(-\frac{\|x\|^2}{4t} - \lambda t\biggr) dt.
\end{align}
It is easy to check that this is well-defined. Moreover, by the dominated convergence theorem, it is easy to see that $K_\lambda$ is a continuous function on its domain. We extend it to whole of $\R^d$ by defining $K_\lambda(0):= \infty$.  We will view $K_\lambda$ as an operator acting on Schwartz functions as
\[
K_\lambda f(x) := \int_{\R^d} K_\lambda(x-y) f(y) dy.
\]
The following lemma, which encapsulates the essential properties of $K_\lambda$, will be proved in Subsection \ref{klemmaproof}.
\begin{lmm}\label{klemma}
The operator $K_\lambda$ is well-defined on $\ms(\R^d)$. Moreover, it is a bijection from $\ms(\R^d)$ onto itself, with $K_\lambda^{-1} = -\Delta +\lambda I$, where $\Delta$ is the Laplacian operator and $I$ is the identity operator. Lastly, $K_\lambda$ commutes with the derivative operators $\partial_i := \partial/\partial x_i$, $i=1,\ldots,d$. 
\end{lmm}
Having defined the operator $K_\lambda$ on $\ms(\R^d)$, we now define an operator  $R_\lambda$ on $\ma(\R^d)$ as
\begin{align}
R_\lambda f &:= \sum_{i=1}^d \biggl(K_\lambda f_i - \lambda^{-1}\sum_{j=1}^d \partial_i \partial_j K_\lambda f_j\biggr) dx_i.\label{rdef}
\end{align}
By Lemma \ref{klemma}, it is clear that $R_\lambda$ maps $\ma(\R^d)$ into itself. The following lemma, proved in Subsection \ref{qlemmaproof}, summarizes the main properties of $R_\lambda$.
\begin{lmm}\label{qlemma}
The operator $R_\lambda$ is a bijection of $\ma(\R^d)$ onto itself, with $R_\lambda^{-1} = Q_\lambda$ given~by
\begin{align}
Q_\lambda f &=\sum_{i=1}^d \biggl(-\Delta f_i+ \lambda f_i +\sum_{j=1}^d \partial_i \partial_j f_j\biggr)dx_i. \label{qdef}
\end{align}
\end{lmm}
We are now ready to define the Euclidean Proca field on $\R^d$.
\begin{defn}\label{mmffdef}
We define the \emph{Euclidean Proca field} on $\R^d$ with parameter $\lambda$ to be a random distributional $1$-form $X$ with the property that for each $f\in \ma(\R^d)$, $X(f)$ is a Gaussian random variable with mean zero and variance 
\begin{align}\label{innproddef}
(f, R_\lambda f) := \int_{\R^d} f(x)\cdot R_\lambda f(x) dx,
\end{align}
where $R_\lambda$ is the operator defined in equation~\eqref{rdef}. 
\end{defn}
The above definition is somewhat opaque. It is easier to understand once we start writing things in the language of differential geometry. Let $d$ denote the exterior derivative operator acting on $\R$-valued differential forms on $\R^d$. The exterior derivative of  a $1$-form $f=\sum_{i=1}^d f_i dx_i$ is the $2$-form
\[
d f = \sum_{1\le i<j\le d} (\partial_i f_j - \partial_j f_i) dx_i \wedge dx_j.
\]
The adjoint of the exterior derivative (sometimes called the exterior coderivative), denoted by $d^*$, is the operator that satisfies $(df, g) = (f, d^*g)$ for any $k$-form $f$ and $(k+1)$-form $g$, for any $0\le k\le d-1$. The  coderivative of $2$-form $g = \sum_{1\le i<j\le d} g_{ij} dx_i \wedge dx_j$ is the $1$-form
\[
d^* g = - \sum_{i=1}^d \biggl(\sum_{j=1}^d \partial_j g_{ij}\biggr) dx_i.
\]
It is easy to check that the Laplacian operator acting on $1$-forms as
\[
\Delta f = \sum_{i=1}^d \Delta f_i dx_i
\]
is equal to $-(d^*d + dd^*)$. On the other hand, for a $1$-form $f$, $d^*f $ is simply $-\sum_{i=1}^d \partial_i f_i$, and therefore,
\[
dd^* f = -\sum_{i,j=1}^d \partial_i \partial_j f_j dx_i.
\]
Thus, the operator $Q_\lambda$ is equal to $-\Delta - dd^*+\lambda I$, which simplifies to $d^*d + \lambda I$.  Since by definition, the Proca field is a Gaussian field with inverse covariance kernel $Q_\lambda$, this means that heuristically, the probability density of the Proca field with respect `Lebesgue measure' on the space of fields is proportional to $\exp(-\frac{1}{2}(f, Q_\lambda f))$, which can be rewritten as
\begin{align}\label{diffrep}
\exp\biggl(-\frac{1}{2}(df, df) -\frac{\lambda}{2}(f,f)\biggr).
\end{align}
With the above interpretation, it is transparent why $R_\lambda = Q_\lambda^{-1}$. Indeed, since $dd = 0$, 
\begin{align*}
Q_\lambda R_\lambda &= (d^*d + \lambda I)(I + \lambda^{-1}dd^* ) K_\lambda \\
&= (d^*d + dd^* + \lambda I) K_\lambda = (-\Delta + \lambda I) K_\lambda = I.
\end{align*}
It is not obvious that the Euclidean Proca field actually exists. The following lemma asserts that it does. It is proved in Subsection \ref{mmffproof}, by showing that a certain sequence of discrete fields converges to the Euclidean Proca field in a scaling limit.
\begin{lmm}\label{mmfflemma}
In any dimension $d\ge 2$ and for any $\lambda >0$, the Euclidean Proca field exists. 
\end{lmm}
Note that by linearity, Definition \ref{mmffdef} determines the joint distribution of the random vector $(X(f_1),\ldots,X(f_k))$ for any $f_1,\ldots,f_k\in\ma(\R^d)$.

\subsection{Scaling and translation of the Euclidean Proca field}
We may intuitively think of any random distributional $1$-form $X$ as a random $1$-form by formally defining $X(x) := X(\delta_x)$, where $\delta_x$ is the Dirac delta at $x$. With this notation, linearity of $X$ implies that we formally have
\[
X(f) = X\biggl(\int_{\R^d} f(x)\delta_x dx \biggr) = \int_{\R^d} X(\delta_x) f(x) dx = \int_{\R^d} X(x)f(x)dx.
\]
The advantage of this formal expansion is that it allows us to talk about scaling and translating $X$. Given $a >0$ and $b\in \R^d$, let us define 
\begin{align}\label{yfromx}
Y(x) := X(ax + b).
\end{align}
This is what we get upon rescaling spacetime (i.e., $\R^d$ in this case) by $a$ and translating by $b$. 
What this means is that for any $f\in \ma(\R^d)$, 
\begin{align*}
Y(f) &= \int_{\R^d} Y(x) f(x) dx \\
&= \int_{\R^d} X(ax + b) f(x) dx \\
&= a^{-d} \int_{\R^d} X(z)f((z-b)/a) dz = a^{-d} X(\tau_{a,b} f),
\end{align*}
where $\tau_{a,b} f$ is the function $g(x):= f((x-b)/a)$. Thus, the rigorous way to define a random distributional $1$-form $Y$ that satisfies \eqref{yfromx} is to define, for each $f\in \ma(\R^d)$, $Y(f):= a^{-d} X(\tau_{a,b} f)$. However, this is unnecessarily complicated and opaque; it makes much more sense to just go with the intuitive formula \eqref{yfromx}, as we do in the following lemma. The proof is in Subsection \ref{scalinglmmproof}. 
\begin{lmm}\label{scalinglmm}
Let $X$ be a Euclidean Proca field with parameter $\lambda$. Take any $a > 0$, $b\in \R^d$, and let $Y(x) := a^{(d-2)/2} X(ax + b)$, in the sense defined above. Then $Y$ is a Euclidean Proca field with parameter $a^2\lambda$.
\end{lmm}
Thus, upon scaling and translation, a Proca field just becomes a Proca field with a different parameter.

\subsection{The mass of the Euclidean Proca field}
A theory is called {\it massive} if it has exponential decay of correlations. The decay exponent is called the mass of the theory. The following lemma, proved in Subsection \ref{masslmmproof}, identifies the asymptotic form of the correlation in the Euclidean Proca field. In particular, it shows that the mass of the Euclidean Proca field with parameter $\lambda$ is equal to $\sqrt{\lambda}$. 

In the following, we say that a map $f:\R^d \to \R$ is a {\it bump function} if it is infinitely differentiable, and zero outside a compact region. For a bump function $f$ and a vector $x\in \R^d$, let $f^x$ denote the translate of $f$ by $x$, that is, $f^x(y) := f(x+y)$. We will say that $f\in \ma(\R^d)$ is a bump function if each component of $f$ is a bump function. For two bump functions $f,g\in \ms(\R^d)$, a vector $u\in S^{d-1}$, and a number $\lambda >0$, we define
\begin{align*}
\Psi(f,g,u,\lambda):=  \iint f(y)g(y-v) e^{-\sqrt{\lambda}u\cdot v} dv dy,
\end{align*}
where the integrals are over $\R^d$. The following lemma is proved in  Subsection \ref{masslmmproof}.
\begin{lmm}\label{masslmm}
Let $X$ be a Euclidean Proca field with parameter $\lambda$. Let $\{x_n\}_{n\ge 1}$ be a sequence in $\R^d$ such that $\|x_n\|\to \infty$ and $\|x_n\|^{-1}x_n \to u\in S^{d-1}$. Then for any two bump functions $f,g\in \ma(\R^d)$,
\begin{align*}
&\lim_{n\to\infty} \frac{\E(X(f) X(g^{x_n})) }{\|x_n\|^{-(d-1)/2} e^{-\sqrt{\lambda}\|x_n\|}} \\
&=  \frac{\lambda^{(d-3)/4}}{2(2\pi)^{(d-1)/2}}\biggl(\sum_{i=1}^d \Psi(f_i, g_i, u, \lambda) + \lambda^{-1}\sum_{i,j=1}^d \Psi(\partial_i f_i, \partial_j g_j, u, \lambda)\biggr). 
\end{align*}
\end{lmm}
The main consequence of the lemma is that the correlation between $X(f)$ and $X(g^x)$ behaves like a constant times $\|x\|^{-(d-1)/2} e^{-\sqrt{\lambda}\|x\|}$ when $\|x\|$ is large, where the constant depends on $f$, $g$, $\lambda$ and the unit vector $\|x\|^{-1} x$. In particular, the Euclidean Proca field with parameter $\lambda$ has mass $\sqrt{\lambda}$.

\section{Main results}
We are now ready to state the two main results of the paper. The first result establishes the limit of $\uone$ Yang--Mills--Higgs theory under a certain scaling of the gauge coupling constant and the Higgs length. The second result gives the analogous statement for $\sutwo$ theory.

\subsection{A scaling limit of $\uone$ Yang--Mills--Higgs theory}\label{scale1}
Recall the gauge field $V$ defined in Subsection \ref{uonesec}, obtained by unitary gauge fixing of the original theory, but in infinite volume. That is, consider any infinite volume Gibbs measure obtained by taking a weak limit of the models on $\Lambda$ under periodic boundary condition, as $L\to \infty$. The result stated below does not depend on the choice of the infinite volume measure, as long as it is a weak limit of finite volume measures under periodic boundary conditions. 

For each edge $e$ in $\Z^d$, $V_e$ is an element of $\uone$, which can be viewed as the unit circle $S^1\subseteq \R^2$. We obtain a projection of $V_e$ on the real line as follows. For each $n$, consider the following variant of stereographic projection of the unit sphere $S^{n-1}$ in  $\R^n$. Let $e_1,\ldots,e_n$ denote the standard basis vectors of $\R^n$. For each $x\in S^{n-1}\setminus \{-e_1\}$, take the line passing through $x$ and $-e_1$, and let $y = (y_1,\ldots,y_n)$ be the unique point at the intersection of this line and the plane $x_1=1$. (That is, $y$ is the stereographic projection of $x$ on the plane $x_1=1$;  see Figure \ref{stereo}.) Let $\sigma_{n-1}(x) := (y_2,\ldots,y_n)\in \R^{n-1}$. Note that $\sigma_{n-1}$ defines a bijection between $S^{n-1}\setminus\{-e_1\}$ and $\R^{n-1}$. 

\begin{figure}
\begin{center}
\begin{tikzpicture}[scale = 2.5]
\draw (-1.5,0) to (1.5,0);
\draw (0,-1.5) to (0,1.5);
\draw (0,0) circle [radius=1];
\draw (1,-1.5) to (1,1.5);
\draw (-1,0) to (1,1);
\node at (.2, -.15) {$(0,0)$};
\draw [fill] (0,0) circle [radius = 0.02];
\draw [fill] (-1,0) circle [radius = 0.02];
\node at (1.2, -.15) {$(1,0)$};
\node at (-1.25, -.15) {$(-1,0)$};
\draw [fill] (1,0) circle [radius = 0.02];
\draw [fill] (1,1) circle [radius = 0.02];
\draw [fill] (.6,.8) circle [radius = 0.02];
\node at (.6, .7) {$x$};
\node at (1.5, 1) {$y = (1,\sigma_1(x))$};
\end{tikzpicture}
\caption{Stereographic projection $\sigma_1:S^1 \to \R$.\label{stereo}}
\end{center}
\end{figure}
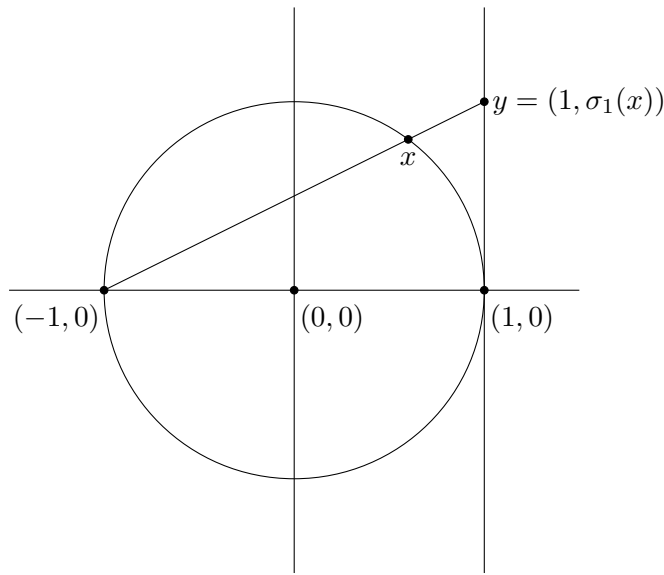

For each $e$, define
\[
A_e :=\frac{1}{g}\sigma_1(V_e). 
\]
Although $\sigma_1$ is undefined at $-e_1$, it is easy to see that the probability that $V_e$ is equal to $-e_1$ is zero. Thus, with probability one, $A_e$ is well-defined for all $e\in E$. 

Given the field $A$, we define a random $1$-form $Y$ as follows. For each $x\in \Z^d$ and each $y$ in the interior of the Voronoi cell of $x$ (i.e., the set of points $z\in \R^d$ such that $x$ is closer to $z$ than any other element of $\Z^d$), and for each $1\le i\le d$, define 
\begin{align*}
Y_i(y):= 
A_{(x, x + e_i)}. 
\end{align*}
On the boundaries of the Voronoi cells, define $Y_i$ in any arbitrary manner. The following theorem identifies the scaling limit of the random $1$-form $Y$ as $g\to 0$ and $\alpha \to \infty$ obeying certain constraints.
\begin{thm}\label{uonethm}
Take any $d\ge 2$ and let $Y$ be the random $1$-form defined above. Define a rescaled version of $Y$ on the lattice $\ve\Z^d$ as 
\[
Z(x) := \ve^{-(d-2)/2} Y(\ve^{-1}x).
\]
Suppose that as $\ve \to 0$, we simultaneously send $\alpha\to \infty$ and $g\to 0$ such that $\alpha g$ remains equal to $c\ve$ for some fixed constant $c$. Further, suppose that $g= O(\ve^{50 d})$ as $\ve \to 0$. Then $Z$ converges in law to the Euclidean Proca field with parameter $c^2$.
\end{thm}
One remarkable aspect of the above result is that the scaling limit only demands that $\alpha g$ be held equal to a constant times the lattice spacing $\ve$. If that is true, and $g$ goes to zero at least as fast as $\ve^{50d}$, then no matter how fast we take $g$ to zero (and therefore, $\alpha$ to infinity), we end up getting the same scaling limit. It seems, therefore, that the behavior of the product $\alpha g$ is what determines the nature of the scaling limit in this model. We ponder on this further in Subsection \ref{opensec}.

\subsection{A scaling limit of $\sutwo$ Yang--Mills--Higgs theory}\label{scale2}
Recall that any element of $\sutwo$ has the form
\[
\begin{pmatrix}
a & b \\
-\bar{b} & \bar{a}
\end{pmatrix},
\]
where $a,b\in \C$ satisfy $|a|^2 + |b|^2 = 1$. Writing $a = x + \I y$ and $b = w+\I z$, we see that any element of $\sutwo$ corresponds to a $4$-tuple of real numbers $(x,y,w,z)$ such that $x^2 + y^2+w^2+z^2 = 1$.  This gives a bijection $\tau: \sutwo \to S^3$. It is not hard to show that the pushforward of the normalized Haar measure under $\tau$ is the uniform probability measure on $S^3$, and conversely, the pushforward of the uniform probability measure on $S^3$ under $\tau^{-1}$ is the normalized Haar measure on $\sutwo$. 

Now recall the gauge field $V$ defined in Subsection \ref{sutwosec}, obtained by unitary gauge fixing of the original theory, but in some infinite volume limit as we did for $\uone$ theory in Subsection \ref{scale1}. Recall the stereographic projection map $\sigma_n: S^n \to \R^n$ defined in Subsection \ref{scale1}. For each edge $e\in E$, define
\[
A_e := \frac{\sqrt{2}}{g} \sigma_3(\tau(V_e)). 
\]
With probability one, $A_e = (A_e^1, A_e^2, A_e^3)$ is a well-defined $\R^3$-valued random vector. Given the field $A$, we define a $3$-tuple of random $1$-forms $Y = (Y^1, Y^2, Y^3)$ as follows. For each $x\in \Z^d$ and each $y$ in the interior of the Voronoi cell of $x$, and for each $1\le i\le d$ and $1\le j\le 3$, define 
\begin{align}\label{yfroma}
Y^j_i(y):= 
A^j_{(x, x + e_i)}. 
\end{align}
On the boundaries of the Voronoi cells, define $Y_i^j$ in any arbitrary manner. The following theorem identifies the scaling limit of $Y$ as $g\to 0$ and $\alpha \to \infty$ under certain constraints.

\begin{thm}\label{sutwothm}
Take any $d\ge 2$ and let $Y$ be the $3$-tuple of random $1$-forms defined above. Define a rescaled version of $Y$ on the lattice $\ve\Z^d$ as 
\begin{align}\label{zfromy}
Z(x) := \ve^{-(d-2)/2} Y(\ve^{-1}x).
\end{align}
Suppose that as $\ve \to 0$, we simultaneously send $\alpha\to \infty$ and $g\to 0$ such that $\alpha g$ remains equal to $c\ve$ for some fixed constant $c$. Further, suppose that $g= O(\ve^{50 d})$ as $\ve \to 0$. Then $Z$ converges in law to a $3$-tuple of independent Euclidean Proca fields with parameter $c^2/2$. 
\end{thm}
Just as in Theorem \ref{uonethm}, notice that here, too, the scaling limit depends on the behavior of the product $\alpha g$ rather than the individual parameters $\alpha$ and $g$.

Notice that in Theorem \ref{sutwothm}, as in Theorem \ref{uonethm}, the mass generation happens only after passing to the scaling limit. Neither of these theorems say anything about exponential decay of correlations {\it before} passing to the scaling limit, i.e., for the models on the lattice.

\subsection{Comparison with prior work}
Scaling limits of $\uone$ theory coupled to a Higgs field have appeared in the prior literature, mostly in the late 70s and early 80s. Some notable contributions are due to \citet*{brydgesetal79}, \citet*{balabanetal84} and \citet*{kennedyking86}, who established mass generation by the Higgs mechanism in various versions of the $\uone$ Yang--Mills--Higgs model on a lattice in $d=2$ and $d=3$. In $d=3$, \citet{balaban82a, balaban82b, balaban83a, balaban83b} proved ultraviolet stability of $\uone$ theory coupled to a Higgs field, which is the first step towards the construction of a continuum limit. For a survey of these and other results, see \citet*{borgsnill87}. The closest cousins of Theorem~\ref{uonethm} in the literature are the results of \citet{gross83} and \citet{driver87}, who proved convergence of the curvature forms of pure $\uone$ theory (i.e., without a Higgs field) to massless versions of the Euclidean Proca field in $d=3$ and $d=4$, respectively.  

In a different setup, \citet{dimock18, dimock20} has proved ultraviolet stability of $\uone$ theory coupled to fermions in $d=3$, using Balaban's approach to renormalization. This is the theory of quantum electrodynamics in $d=3$.  A key difference between  these works and Theorem~\ref{uonethm} is that it is the first result that obtains a massive field in the scaling limit, in a dimension greater than two. All of the above results are either in $d=2$, or obtain a massless field in $d\ge 3$. 

Similarly, Theorem \ref{sutwothm} is the first construction of a massive scaling limit of a non-Abelian lattice Yang--Mills theory in dimensions bigger than two. 
For prior progress on constructions of scaling limits of non-Abelian Yang--Mills theories in $d\ge 3$, see the works of \citet{balaban85, balaban89} and \citet*{magnenetal93}. There has been a resurgence of interest in this topic in recent years, with contributions from the probability community. An approach to constructing Euclidean Yang--Mills--Higgs theories via rigorous stochastic quantization was initiated by \citet{shen21} and \citet*{chandraetal22, chandraetal22b}. For a survey of this rapidly evolving area, see~\cite{chevyrev22}. For a different approach based on regularization by Yang--Mills heat flows, see~\citet{gross22} and \citet{caochatterjee21, caochatterjee23}. None of these approaches have yet reached their ultimate goal of constructing non-Abelian Euclidean Yang--Mills theories in  $d\ge 3$, but the case looks promising.

In $d=2$, on the other hand, we now have a fairly complete understanding of scaling limits of non-Abelian Yang--Mills theories, due to the contributions of many authors, such as~\citet{driver89}, \citet*{grossetal89}, \citet{fine91}, \citet{sengupta97}, \citet{levy03} and \citet{chevyrev19}. More recently, \citet{chevyrevshen23} have characterized the 2D Yang--Mills measure as the unique invariant measure of the Langevin dynamic; this   is important for universality of scaling limits. 

However, the construction of Yang--Mills--Higgs theories in the continuum remains open even in 2D. The most important recent progress in this direction is the work of \citet{bringmanncao24}, which proves global well-posedness of the Abelian Yang--Mills--Higgs heat flow in 2D, essentially giving a dynamical construction of Abelian Yang--Mills--Higgs theory in 2D.

Two recent works that are analogous to Theorem \ref{sutwothm} in that they obtained Gaussian limits of discrete models with non-Abelian symmetries are those of \citet*{aruetal22} and \citet{giulianiott23}. These papers look at the spin $O(N)$ models in $d=2$ and obtain Gaussian scaling limits. Unlike Theorem \ref{sutwothm}, however, the scaling limits obtained in \cite{aruetal22, giulianiott23} are non-massive. The proof sketch provided in the next subsection will show the Higgs mechanism generates the mass in the scaling limits in Theorem \ref{uonethm} and Theorem~\ref{sutwothm}.

Proving mass gap in Yang--Mills theories is a fundamental open question in $d\ge 3$. For pure gauge theories on the lattice (i.e., no Higgs field and no scaling limit), mass gap was rigorously proved at strong coupling (large $g$) by \citet{osterwalderseiler78} using the cluster expansion technique of \citet*{glimmetal76}. Similar arguments were used by \citet{seiler82} to prove mass gap in Yang--Mills--Higgs models on lattices when $\alpha g$ is large enough. An important recent work on mass gap is the paper by \citet*{shenetal23}, which
improves the large $g$ condition in \cite{osterwalderseiler78}. In particular, \cite{shenetal23} allows $g$ to be chosen uniformly in $N$ for $\sun$ theory. 

The problem with holding $\alpha g$ larger than a constant, however, is that the correlation length does not diverge to infinity~\cite{seiler82} and therefore one cannot pass to a continuum limit. In this sense, Theorem \ref{sutwothm} is the first rigorous proof of the validity of the Higgs mechanism in the continuum limit of a Euclidean Yang--Mills theory in a dimension bigger than two. It also gives a partial explanation for the phase diagrams constructed by \citet{fradkinshenker79}, \citet{banksrabinovici79} and many subsequent authors in the physics literature, which suggest that a mass gap can be present even if $\alpha g$ is very small. 

In \citet[Fig.~2]{fradkinshenker79}, for instance, the phase diagram for the $\uone$ model on the lattice (without taking a scaling limit) indicates that correlations decay exponentially when either $g$ is bigger than some threshold $g_0$ (which is already provable rigorously from the results of \cite{osterwalderseiler78}), or if $\alpha^2$ is bigger than some concave function of $1/g^2$ growing slower than linearly. In particular, this phase diagram (if true) implies that the region of exponential decay of correlations contains values of $(\alpha,g)$ such that the product $\alpha g$ is smaller than any given constant. This goes beyond the `strong coupling' region given by \citet{seiler82} where $\alpha g$ is bigger than a fixed constant. It is difficult, however, to compare such phase diagrams with our results, since we obtain exponential decay of correlations only in the continuum scaling limit. But we do take $\alpha g$ to zero while taking this limit, which gives partial support for the predictions from the physics papers.

\subsection{Sketch of the proof}\label{sketchsec}
The following is a sketch of the proof of Theorem \ref{sutwothm}. Throughout the sketch, $C_1,C_2,\ldots$ will denote constants that do not vary with $\ve$ (and in particular, do not depend on $L$, $\alpha$, or $g$). Also, let us keep $L$ fixed in the beginning instead of being sent to infinity. Let $\Sigma:= \sutwo^E$. Let $\Sigma'$ be the subset of $\Sigma$ consisting of all $U$ such that $\|U_e - I\|\le \alpha^{-1}$ for all $e$, where $\|\cdot\|$ denotes Euclidean (i.e., Frobenius) norm. It is not hard to see that the normalized Haar measure of $\Sigma'$ is  $\ge (C_1\alpha)^{-C_2L^d}$. 

Now, recall the field $V$ defined in Section \ref{sutwosec}. In Lemma \ref{keylmm}, we obtain the key estimate 
\begin{align*}
\ee\|V_e - I\|^2 &\le \frac{C_1}{\alpha^4 g^2} + \frac{C_2\log \alpha}{\alpha^2}. 
\end{align*}
Note that this bound {\it has no dependence on $L$.} This is the crucial fact that allows us to first send $L\to \infty$ and then construct the continuum limit, as follows.

Note that since the above bound has no dependence on $L$, we may as well send $L\to\infty$ and still retain the bound. Let $\Lambda' := \{-M,\ldots,M\}^d$, where $M$ will be chosen later (to grow as $\ve \to 0$). Let $E'$ denote the set of edges of $\Lambda'$. Then 
\begin{align*}
\pp(\max_{e\in E'}\|V_e - I\| > \delta) &\le \frac{1}{\delta^2}\sum_{e\in E'}\ee\|V_e-I\|^2\\
&\le \frac{C_8M^d}{\alpha^4 g^2\delta^2} + \frac{C_9M^d \log \alpha}{\alpha^2\delta^2}. 
\end{align*}
Choosing $\alpha = g^{\kappa-1}$, $M=g^{-4\kappa}$ and $\delta= g^{1-6d\kappa}$ for some sufficiently small $\kappa$, this bound can be taken to zero as $g\to 0$. Armed with this and several other estimates of a similar nature, the rest of the analysis proceeds by perturbative expansion around a Gaussian measure, which scales to the Euclidean Proca field. 

Roughly speaking, the formula for the probability density function of $V$ shows that if $V_e = I + g A_e + o(g)$, then the probability density function of the matrix-valued field $A$ is proportional to 
\[
\exp\biggl(-\frac{1}{2}\sum_{p\in P}\|A_p\|^2 - \frac{\alpha^2g^2}{4}\sum_{e\in E} \|A_e\|^2+\text{ lower order terms}\biggr),
\]
where $A_p$ is the sum of $A_e$'s around the plaquette $p$ with appropriate signs. If we set the lattice spacing to be $\ve:=\alpha g$, and $g$ goes to zero sufficiently fast as $\ve \to 0$, then the lower order terms in the above display do not interfere in the final outcome, and the field converges to a massive Proca field. This is because the Proca field is the scaling limit of a discrete field $B$ on $\Z^d$ with probability density function proportional to 
\[
\exp\biggl(-\frac{1}{2}\sum_{p\in P}\|B_p\|^2 - \frac{\ve^2}{4}\sum_{e\in E} \|B_e\|^2\biggr).
\]
The three components of $A$ decouple because the couplings are all in the lower order terms, arising due to the non-Abelian nature of $\sutwo$. Sending $g\to 0$ sufficiently fast suppresses this non-Abelian interference.

Lastly, let us remark on why the same technique does not work for non-Abelian gauge groups other than $\sutwo$. Recall that we showed in Subsection \ref{sutwosec} that unitary gauge fixing in $\sutwo$ theory eliminates the Higgs field completely, and leaves us with only a gauge field to study. The term in the action corresponding to the interaction of the Higgs field and the gauge field gets replaced by a mass term for the gauge field, as shown by the formula for the probability density function. One cannot employ such a trick for $\mathrm{SU}(3)$, for instance, when the Higgs field is either in the fundamental or the adjoint representation. There seems to be no simple gauge fixing procedure that transforms the problem to that of studying only a gauge field where a mass term appears in the action. The issue is that even after gauge fixing, a nontrivial Higgs field remains, because the stabilizer of the gauge action is nontrivial.

\subsection{Open questions}\label{opensec}
It is evident from the discussion in the previous subsection that holding $\alpha g = c \ve$ is the key fact that gives rise to a massive field in the scaling limit. It is natural to ask whether sending $g$ to zero so fast that $g= O(\ve^{50d})$ is really necessary for obtaining this limit. The answer may depend on the model. For $\uone$ theory, it is possible that simply under the conditions that $\alpha g = c \ve$ and $g \to 0$ as $\ve \to 0$, with no further conditions on the decay rate of $g$ (or perhaps a very slow decay rate), we obtain the scaling limit given by Theorem \ref{uonethm}. For $\sutwo$ theory, it seems likely that a slow decay of $g$ will cause non-Abelian effects to manifest themselves, possibly leading to a different scaling limit. It would be interesting to figure out the various regimes for both theories. But holding $\alpha g = c\ve$ seems to be a necessary condition for mass generation in all cases.

Another open problem is to extend the results to general Higgs potentials, going beyond the degenerate potential considered in this paper. It is not clear how to make this extension.

\section{Proofs for the Euclidean Proca field}\label{proofsec}
All proofs for the results about the Euclidean Proca field are in this section, presented in the same order as  the respective results have appeared in previous sections. 
\subsection{Proof of Lemma \ref{klemma}}\label{klemmaproof}
We begin with a preliminary upper bound on $K_\lambda(x)$. 
\begin{lmm}\label{klmm}
The function $K_\lambda$ satisfies the bound
\begin{align*}
K_\lambda(x)\le 
\begin{cases} 
C(d)e^{-\|x\|\sqrt{\lambda/2}}\|x\|^{-(d-2)} &\text{ if } d\ge 3,\\
C(d)e^{-\|x\|\sqrt{\lambda}/4} (1+\lambda^{-1})&\text{ if } d=2 \text{ and } \|x\|\ge 2\sqrt{\lambda},\\
C(d)e^{-\|x\|\sqrt{\lambda}/4}(\lambda^{-1} + \log(2\sqrt{\lambda}/\|x\|)) &\text{ if } d=2 \text{ and } \|x\| < 2\sqrt{\lambda},
\end{cases}
\end{align*}
for all $x\in \R^d\setminus \{0\}$, where $C(d)$ is a constant that depends only on $d$.
\end{lmm}
\begin{proof}
First, let $d\ge 3$. Making the change of variable $s = \|x\|^2/4t$ in the integral defining $K_\lambda(x)$ in equation~\eqref{klambdadef} shows that
\begin{align}\label{kalt}
K_\lambda(x) &= \frac{1}{\|x\|^{d-2}}\int_0^\infty \frac{s^{(d-4)/2}}{4\pi^{d/2}}\exp\biggl(-s- \frac{\lambda \|x\|^2}{4s}\biggr) ds.
\end{align}
Now note that by the AM-GM inequality, 
\begin{align*}
\frac{s}{2}+ \frac{\lambda \|x\|^2}{4s}&\ge \biggl(\frac{s \lambda \|x\|^2}{2s}\biggr)^{1/2} = \frac{\|x\|\sqrt{\lambda}}{\sqrt{2}}.  
\end{align*}
By \eqref{kalt}, this shows that
\begin{align*}
K_\lambda(x) &\le \frac{e^{-\|x\|\sqrt{\lambda/2}}}{\|x\|^{d-2}}\int_0^\infty \frac{s^{(d-4)/2}}{4\pi^{d/2}}e^{-s/2} ds.
\end{align*}
It is easy to see that the above integral is finite, since $d\ge 3$. This completes the proof when $d\ge 3$. Next, let $d=2$. First, note that by an application of the AM-GM inequality similar to the above, we get
\begin{align*}
K_\lambda(x) &\le e^{-\|x\|\sqrt{\lambda}/4}\int_0^\infty \frac{1}{4\pi t} \exp\biggl(-\frac{\|x\|^2}{8t} - \frac{\lambda t}{2}\biggr) dt.
\end{align*}
We break up the above integral into two parts, from $0$ to $\|x\|/2\sqrt{\lambda}$, and from $\|x\|/2\sqrt{\lambda}$ to infinity. For the second part, note that
\begin{align*}
\int_{\|x\|/2\sqrt{\lambda}}^\infty \frac{1}{4\pi t} \exp\biggl(-\frac{\|x\|^2}{8t} - \frac{\lambda t}{2}\biggr) dt&\le \int_{\|x\|/2\sqrt{\lambda}}^\infty \frac{1}{4\pi t} e^{-\lambda t/2} dt.
\end{align*}
If $\|x\|\ge 2\sqrt{\lambda}$, then the right side is bounded above by $C/\lambda$ for a universal constant $C$. If $\|x\|<2\sqrt{\lambda}$, we further break up the integral into an integral from $\|x\|/2\sqrt{\lambda}$ and an integral from $1$ to $\infty$. The second integral is bounded above by $C/\lambda$, and the first integral is bounded above $C\log(2\sqrt{\lambda}/\|x\|)$. Next, note that by the change of variable $s = \|x\|^2/4t$, we get
\begin{align*}
\int_0^{\|x\|/2\sqrt{\lambda}} \frac{1}{4\pi t} \exp\biggl(-\frac{\|x\|^2}{8t} - \frac{\lambda t}{2}\biggr) dt &= 
\int_{\|x\|/2\sqrt{\lambda}}^\infty \frac{1}{4\pi s}\exp\biggl(-\frac{s}{2}- \frac{\lambda \|x\|^2}{8s}\biggr) ds\\
&\le \int_{\|x\|/2\sqrt{\lambda}}^\infty \frac{1}{4\pi s} e^{-s/2} ds.
\end{align*}
By a similar argument as above, this is bounded by  a universal constant $C$ if $\|x\|\ge 2\sqrt{\lambda}$ and by $C\log(2\sqrt{\lambda}/\|x\|)$ if $\|x\|< 2\sqrt{\lambda}$. This completes the proof in the case $d=2$.
\end{proof}

The bound from Lemma \ref{klmm} shows that $K_\lambda f$ is well-defined and finite for any $f\in \ms(\R^d)$. The next lemma proves some of the properties asserted by Lemma \ref{klemma}.
\begin{lmm}\label{kclosedlmm}
The operator $K_\lambda$ sends Schwartz functions to Schwartz functions, and commutes with the derivative operators $\partial_i := \partial/\partial x_i$, $i=1,\ldots,d$. 
\end{lmm}
\begin{proof}
Take any $f\in \ms(\R^d)$. Note that $K_\lambda f$ can be alternatively expressed as
\begin{align*}
K_\lambda f(x) = \int_{\R^d} K_\lambda(y) f(x-y) dy.
\end{align*}
Since $f\in \ms(\R^d)$, there is a finite constant $C$ such that for any $x, y, z\in \R^d$,
\begin{align*}
|f(x+z-y)-f(x-y)| &\le C\|z\|. 
\end{align*}
By the bound from Lemma \ref{klmm}, $K_\lambda$ is integrable on $\R^d$. From these two observations, it follows by the dominated convergence theorem that $K_\lambda f$ is a differentiable function, and for each $i$, 
\[
\partial_i K_\lambda f = K_\lambda \partial_i f.
\]
This proves, in particular, that $K_\lambda$ commutes with the derivative operators. By iteration, this also proves that $K_\lambda f$ is infinitely differentiable for any $f\in \ms(\R^d)$, and $K_\lambda$ commutes with any product of derivative operators. 

Thus, to show that $K_\lambda$ maps $\ms(\R^d)$ into $\ms(\R^d)$, it suffices to prove that for any rapidly decaying function $f\in C^\infty(\R^d)$, $K_\lambda f$ is also rapidly decaying. Take any such $f$, and any $\alpha >0$. By the rapidly decaying nature of $f$, we know that there is some $C$ such that $|f(x)|\le C(1+\|x\|)^{-\alpha}$ for all $x$. Suppose that $d\ge 3$. Then by Lemma \ref{klmm}, we have
\begin{align*}
|K_\lambda f(x)| &\le \int_{\R^d} K_\lambda(y) |f(x-y)|dy \\
&\le C_1 \int_{\R^d} \frac{e^{-\|y\|\sqrt{\lambda/2}}}{\|y\|^{d-2}} (1+\|x-y\|)^{-\alpha} dy,
\end{align*}
where $C_1$ depends only on $d$ and $f$. Now, for $\|y\|\le \|x\|/2$, 
\[
\frac{e^{-\|y\|\sqrt{\lambda/2}}}{\|y\|^{d-2}} (1+\|x-y\|)^{-\alpha} \le \frac{e^{-\|y\|\sqrt{\lambda/2}}}{\|y\|^{d-2}} (1+\|x\|/2)^{-\alpha},
\]
and for $\|y\|> \|x\|/2$,
\[
\frac{e^{-\|y\|\sqrt{\lambda/2}}}{\|y\|^{d-2}} (1+\|x-y\|)^{-\alpha}\le \frac{e^{-\|y\|\sqrt{\lambda/2}}}{\|y\|^{d-2}}.
\]
Using these two bounds in these two regions, and passing to polar coordinates in the second integral below, we get
\begin{align*}
|K_\lambda f(x)| &\le  C_2 (1+\|x\|/2)^{-\alpha} \int_{\|y\|\le \|x\|/2} \frac{e^{-\|y\|\sqrt{\lambda/2}}}{\|y\|^{d-2}} dy + C_3 \int_{\|y\|> \|x\|/2} \frac{e^{-\|y\|\sqrt{\lambda/2}}}{\|y\|^{d-2}}dy\\
&\le C_4 (1+\|x\|/2)^{-\alpha} \int_{\R^d} \frac{e^{-\|y\|\sqrt{\lambda/2}}}{\|y\|^{d-2}} dy + C_5 \int_{\|x\|/2}^\infty r e^{-r\sqrt{\lambda/2}} dr \\
&\le C_6 (1+\|x\|/2)^{-\alpha} + C_7 e^{-C_5\|x\|},
\end{align*}
where $C_2, \ldots,C_7$ depend only on $d$, $f$ and $\lambda$. This shows that $K_\lambda f$ is rapidly decaying if $d\ge 3$. For $d=2$, all steps go through exactly as above, using the bound from Lemma \ref{klmm}.
\end{proof}

The next lemma proves the remaining properties of $K_\lambda$ asserted by Lemma \ref{klemma}, thereby completing the proof of Lemma \ref{klemma}.
\begin{lmm}\label{kinvlmm}
The operator $K_\lambda$ is a bijection of $\ms(\R^d)$ onto itself, and its inverse is the operator $-\Delta +\lambda I$, where $\Delta$ is the Laplacian operator and $I$ is the identity operator.
\end{lmm}
\begin{proof}
Take any $f\in \ms(\R^d)$ and let $g := K_\lambda f$. By Lemma \ref{kclosedlmm}, $g$ is a Schwartz function, and $\Delta g = K_\lambda \Delta f$. Thus, by Fubini's theorem,
\begin{align*}
\Delta g(x) &= \int_{\R^d}K(x-y) \Delta f(y) dy\\
&= \int_{\R^d} \int_0^\infty  \frac{1}{(4\pi t)^{d/2}} \exp\biggl(-\frac{\|x-y\|^2}{4t} - \lambda t\biggr)\Delta f(y) dt dy\\
&=  \int_0^\infty  \int_{\R^d} \frac{1}{(4\pi t)^{d/2}} \exp\biggl(-\frac{\|x-y\|^2}{4t} - \lambda t\biggr)\Delta f(y) dy dt.
\end{align*}
Note that for any $t>0$,
\begin{align*}
&\biggl| \int_{\R^d} \frac{1}{(4\pi t)^{d/2}} \exp\biggl(-\frac{\|x-y\|^2}{4t} - \lambda t\biggr)\Delta f(y) dy \biggr| \\
&\le \|\Delta f\|_\infty  \int_{\R^d} \frac{1}{(4\pi t)^{d/2}} \exp\biggl(-\frac{\|x-y\|^2}{4t}\biggr) dy = \|\Delta f\|_\infty,
\end{align*}
where $\|\Delta f\|_\infty$ denotes the supremum norm of $\Delta f$. This uniform bound shows that
\begin{align*}
\Delta g(x) &= \lim_{\ve \downarrow 0}  \int_\ve^\infty  \int_{\R^d} \frac{1}{(4\pi t)^{d/2}} \exp\biggl(-\frac{\|x-y\|^2}{4t} - \lambda t\biggr)\Delta f(y) dy dt.
\end{align*}
Take any $t>0$. Applying integration by parts (which is valid since the integrand is sufficiently well-behaved and $\Delta f$ is rapidly decaying), we get
\begin{align*}
&\int_{\R^d} \exp\biggl(-\frac{\|x-y\|^2}{4t}\biggr)\Delta f(y)dy \\
&= \int_{\R^d}  f(y) \sum_{i=1}^d \frac{\partial^2}{\partial y_i^2}\exp\biggl(-\frac{\|x-y\|^2}{4t}\biggr) dy\\
&= \int_{\R^d} f(y) \biggl(-\frac{d}{2t} + \frac{\|x-y\|^2}{4t^2}\biggr) \exp\biggl(-\frac{\|x-y\|^2}{4t}\biggr) dy.
\end{align*}
Plugging this into the previous display, and again applying Fubini's theorem, we get
\begin{align*}
\Delta g(x) 
&=\lim_{\ve \downarrow 0} \int_{\R^d} \int_\ve^\infty   \frac{1}{(4\pi t)^{d/2}} f(y) \biggl(-\frac{d}{2t} + \frac{\|x-y\|^2}{4t^2}\biggr) \exp\biggl(-\frac{\|x-y\|^2}{4t}-\lambda t\biggr) dt dy.
\end{align*}
Now note that for any distinct $x$ and $y$,
\begin{align*}
&\frac{\partial}{\partial t} \biggl(\frac{1}{(4\pi t)^{d/2}} \exp\biggl(-\frac{\|x-y\|^2}{4t}\biggr)\biggr) \\
&=  \frac{1}{(4\pi t)^{d/2}}  \biggl(-\frac{d}{2t} + \frac{\|x-y\|^2}{4t^2}\biggr) \exp\biggl(-\frac{\|x-y\|^2}{4t}\biggr).
\end{align*}
Using this in the previous display, we get
\begin{align*}
\Delta g(x) &= \lim_{\ve \downarrow 0} \int_{\R^d}\int_\ve^\infty f(y)e^{-\lambda t} \frac{\partial}{\partial t}  \biggl(\frac{1}{(4\pi t)^{d/2}} \exp\biggl(-\frac{\|x-y\|^2}{4t}\biggr)\biggr) dtdy.
\end{align*}
Take any $\ve >0$. Using integration by parts, we get
\begin{align*}
&\int_\ve^\infty e^{-\lambda t} \frac{\partial}{\partial t}  \biggl(\frac{1}{(4\pi t)^{d/2}} \exp\biggl(-\frac{\|x-y\|^2}{4t}\biggr)\biggr) dt \\
&= - e^{-\lambda \ve} \frac{1}{(4\pi \ve )^{d/2}} \exp\biggl(-\frac{\|x-y\|^2}{4\ve}\biggr) + \int_\ve^\infty \lambda e^{-\lambda t} \frac{1}{(4\pi t)^{d/2}} \exp\biggl(-\frac{\|x-y\|^2}{4t}\biggr) dt.
\end{align*}
Combining this with the previous step gives us
\begin{align}\label{deltagfinal}
\Delta g(x) &= - \lim_{\ve \downarrow 0} \int_{\R^d} f(y) e^{-\lambda \ve} \frac{1}{(4\pi \ve )^{d/2}} \exp\biggl(-\frac{\|x-y\|^2}{4\ve}\biggr) dy \notag \\
&\qquad + \lim_{\ve \downarrow 0} \int_{\R^d} \int_\ve^\infty \lambda  f(y) e^{-\lambda t} \frac{1}{(4\pi t)^{d/2}} \exp\biggl(-\frac{\|x-y\|^2}{4t}\biggr) dt dy.
\end{align}
By the change of variable $z = \ve^{-1/2}(x-y)$, we have
\begin{align*}
\int_{\R^d} f(y) \frac{1}{(4\pi \ve )^{d/2}} \exp\biggl(-\frac{\|x-y\|^2}{4\ve}\biggr) dy &= \int_{\R^d} f(x - \ve^{1/2} z)  \frac{1}{(4\pi )^{d/2}} e^{-\|z\|^2/4} dz. 
\end{align*}
By the dominated convergence theorem, this shows that the first limit in \eqref{deltagfinal} is equal to $f(x)$. Again, by the dominated convergence theorem and Fubini's theorem, the second limit equals $\lambda g(x)$. Thus, $\Delta g(x)=-f(x)+\lambda g(x)$. This proves that $(-\Delta +\lambda I)K_\lambda =I$. Since $K_\lambda$ commutes with differential operators, this also proves that $K_\lambda(-\Delta +\lambda I) = I$. These two identities imply that $K_\lambda$ is a bijection of $\ms(\R^d)$ onto itself, and $K_\lambda^{-1} = - \Delta + \lambda I$.
\end{proof}

\subsection{Proof of Lemma \ref{qlemma}}\label{qlemmaproof}
It is clear that $Q_\lambda$ maps $\ma(\R^d)$ into itself, and by Lemma \ref{kclosedlmm}, the same is true for $R_\lambda$. Now take any $f\in \ma(\R^d)$ and let $g:= R_\lambda f$ and $h:=Q_\lambda g$. Then for $1\le i\le d$, using Lemma~\ref{kclosedlmm} and Lemma~\ref{kinvlmm}, we get
\begin{align*}
h_i &= -\Delta g_i + \lambda g_i + \sum_{j=1}^d \partial_i\partial_j g_j\\
&= (-\Delta+\lambda I)\biggl(K_\lambda f_i -\lambda^{-1}\sum_{j=1}^d \partial_i\partial_j K_\lambda f_j\biggr)  + \sum_{j=1}^d \partial_i \partial_j\biggl(K_\lambda f_j -\lambda^{-1}\sum_{k=1}^d \partial_j\partial_k K_\lambda f_k\biggr)\\
&= f_i -\lambda^{-1}\sum_{j=1}^d \partial_i\partial_j f_j + \sum_{j=1}^d \partial_i \partial_j K_\lambda f_j - \lambda^{-1}\sum_{j,k=1}^d \partial_i \partial_j^2 \partial_k K_\lambda f_k\\
&= f_i -\lambda^{-1}\sum_{j=1}^d \partial_i\partial_j f_j + \lambda^{-1}\sum_{k=1}^d \partial_i \partial_k K_\lambda(\lambda I- \Delta)f_k \\
&= f_i. 
\end{align*}
Thus, $Q_\lambda R_\lambda = I$. By Lemma \ref{kclosedlmm}, $Q_\lambda$ and $R_\lambda$ commute. Thus, $R_\lambda Q_\lambda = I$. This proves that $Q_\lambda$ and $R_\lambda$ are both bijections of $\ma(\R^d)$ onto itself, and $Q_\lambda = R_\lambda^{-1}$.

\subsection{The discrete Proca field and the proof of Lemma \ref{mmfflemma}}\label{mmffproof}
We will prove Lemma \ref{mmfflemma} by showing that a certain sequence of random $1$-forms converges in law to the Euclidean Proca field, in the sense that each coordinate converges in law to the required limit. The following lemma shows that this is sufficient for proving the existence of the field. 

\begin{lmm}\label{existlmm}
Let $\{X_n\}_{n\ge 1}$ be a sequence of random distributional $1$-forms such that for each $f\in \ma(\R^d)$, $X_n(f)$ converges in law as $n\to\infty$. Then there is a random distributional $1$-form $X$ such that $X_n$ converges in law to $X$.
\end{lmm}
\begin{proof}
Linearity implies that for any $f_1,\ldots,f_k$, $(X_n(f_1),\ldots,X_n(f_k))$ converges in law. This determines a consistent family of finite dimensional laws on $\R^{\ma(\R^d)}$. By the Kolmogorov consistency theorem, this implies the existence of $X=(X(f))_{f\in \ma(\R^d)}$ such that for any $k$ and any $f_1,\ldots, f_k$, the random vector $(X_n(f_1),\ldots,X_n(f_k))$ converges in law to $(X(f_1),\ldots, X(f_k))$. In particular, since for any given $a,b\in \R$ and $f,g\in \ma(\R^d)$, $X_n(a f+b g) = a X_n(f) + b X_n(g)$ with probability one for each $n$, it follows that $X(\alpha f+\beta g) = \alpha X(f)+\beta X(g)$ with probability one. Thus, $X$ is a random distributional $1$-form and $X_n$ converges in law to $X$ as $n\to\infty$.
\end{proof}

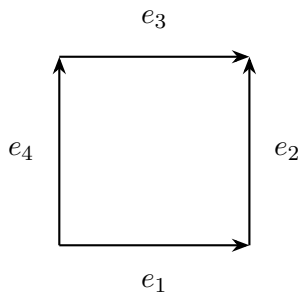
\begin{figure}
\begin{center}
\begin{tikzpicture}[scale = 2.5]
\draw [-Stealth, thick] (0,0) to (1,0);
\draw [-Stealth, thick] (1,0) to (1,1);
\draw [Stealth-, thick] (1,1) to (0,1);
\draw [Stealth-, thick] (0,1) to (0,0);
\node at (.5,-.2) {$e_1$};
\node at (1.2,.5) {$e_2$};
\node at (.5,1.2) {$e_3$};
\node at (-.2,.5) {$e_4$};
\end{tikzpicture}
\caption{A plaquette bounded by four positively oriented edges.\label{plaquette2}}
\end{center}
\end{figure}

Let $\Lambda$, $E$ and $P$ be as in Subsection \ref{lgtsec}, but now {\it without} identifying opposite faces. For a vector $x = (x(e))_{e\in E}\in \R^E$ and a plaquette $p\in P$, we define the number $x(p)$ as follows. 
Let $e_1,e_2,e_3,e_4$ be the edges of $p$, numbered such that the left endpoint of $e_1$ is the smallest vertex of $p$ in the lexicographic ordering, $e_4$ is incident to the left endpoint of $e_1$, and $e_2$ is incident to the right endpoint of $e_1$ (see Figure \ref{plaquette2}, and notice the difference with Figure \ref{plaquette1}). Then we define
\begin{align}\label{xpdef}
x(p) := x(e_1)+x(e_2)-x(e_3)-x(e_4).
\end{align}
\begin{defn}\label{dmfdef}
We define the \emph{discrete Proca field} on $\Lambda$ with parameter $\ve^2$ and free boundary condition to be the Gaussian random vector $X = (X(e))_{e\in E}$ with probability density proportional to 
\[
\exp\biggl(-\frac{1}{2}\sum_{p\in P} x(p)^2 -\frac{\ve^2}{2} \sum_{e\in E} x(e)^2\biggr).
\]
\end{defn}
Note that this is a discretization of the continuum Proca field, with $\lambda$ replaced by $\ve^2$, by the differential geometric representation of the density of the Proca field in equation~\eqref{diffrep}. 

Given $X$, we define a random $1$-form $Y$ as follows. For each $v\in \ve \Z^d$ and each $w$ in the interior of the Voronoi cell of $v$ (i.e., the set of points $x\in \R^d$ such that $v$ is closer to $x$ than any other element of $\ve \Z^d$), and for each $1\le i\le d$, define 
\begin{align}\label{yfromx2}
Y_i(w):= 
\begin{cases}
\ve^{-(d-2)/2}X(\ve^{-1} v, \ve^{-1}v + e_i) &\text{ if } (\ve^{-1}v, \ve^{-1}v + e_i) \in E,\\
0 &\text{ if not.}
\end{cases}
\end{align}
The following theorem shows that $Y$ converges in law to the Euclidean Proca field in a certain scaling limit. In particular, it shows that the Euclidean Proca field exists, thereby proving Lemma \ref{mmfflemma}. We will also need for other purposes later.
\begin{thm}\label{convthm}
Suppose that we take $L\to\infty$ and $\ve \to 0$ along a sequence satisfying $L \ge \ve^{-1-\delta}$ for some $\delta>0$.  Then for any $f\in \ma(\R^d)$,  $Y(f)$ converges in law to a Gaussian random variable with mean zero and variance $(f, R_1f)$, where $R_1$ is the operator defined in equation \eqref{rdef} with $\lambda=1$. Thus, the Euclidean Proca field exists for any $\lambda >0$, and $Y$ converges in law to the Euclidean Proca field with $\lambda = 1$.
\end{thm}
\begin{proof}
Take any $g\in \ma(\R^d)$. Throughout this proof, $O(\ve^k)$ will denote any quantity whose magnitude is bounded above by $C\ve^k$ for some constant $C$ that depends only on $g$, $d$ and $\delta$. Let $\nabla g$ denote the gradient of $g$, $Hg$ denote the Hessian of $g$, and $\Delta g$ denote the Laplacian of $g$. Fixing $\ve$ and $L\ge  \ve^{-1-\delta}$, let 
\begin{align}\label{rdefinition}
R(e,e') := \ve^{-(d-2)}\E(X(e) X(e')). 
\end{align}
We will say that two edges $e,e'\in E$ are neighbors if they belong to a common plaquette. If $e$ and $e'$ are neighbors belonging to a common plaquette $p$, then we will say that they are positive neighbors if they are either the first two edges of $p$ or the last two edges of $p$, when the edges are ordered as in the definition \eqref{xpdef} of $x(p)$. Otherwise, we will say that they are negative neighbors. 
Define $Q\in \R^{E\times E}$ as
\begin{align}\label{qdefinition}
Q(e,e') &:= 
\begin{cases}
(2d-2)\ve^{d-2} +\ve^d &\text{ if } e=e',\\
2 \ve^{d-2} &\text{ if $e$ and $e'$ are positive neighbors,}\\
-2\ve^{d-2} &\text{ if $e$ and $e'$ are negative neighbors.}
\end{cases}
\end{align}
It is easy to check that for any $x\in \R^E$,
\begin{align*}
x^*Qx &= \ve^{d-2}\sum_{p\in P} x(p)^2 +\ve^d \sum_{e\in E} x(e)^2.
\end{align*}
This shows that for any $x\in \R^E$,
\begin{align*}
x^*Qx &\ge \ve^d \|x\|^2.
\end{align*}
In particular, $Q$ is invertible, and the smallest eigenvalue of $Q$ is at least $\ve^d$. It is clear that $R = Q^{-1}$. Thus, the largest eigenvalue of $R$ is at most $\ve^{-d}$. 

For $e = (a,a+e_i)\in E$ (where now, and in the following, $e_1,\ldots,e_d$ denote the standard basis vectors of $\R^d$), let $x(e) := g_i(\ve a)$. Let $z := Qx$. Note that 
\begin{align*}
z(e) &= ((2d-2)\ve^{d-2} + \ve^d) x(e) + 2\ve^{d-2}\sum_{e'\in N^+(e)} x(e') - 2\ve^{d-2}\sum_{e'\in N^-(e)} x(e'),
\end{align*}
where $N^+(e)$ is the set of positive neighbors of $e$ and $N^{-}(e)$ is the set of negative neighbors of~$e$. 

Now suppose that $e = (a,a+e_i)$ is not a boundary edge of $\Lambda$. Then the set of positive neighbors of $e$ consists of the edges $(a+e_i, a+e_i + e_j)$ and $(a-e_j, a)$ as $j$ runs over $\{1,\ldots,d\} \setminus\{i\}$. The set negative neighbors of $e$ consists of the edges $(a\pm e_j, a\pm e_j +e_i)$ and the edges $(a, a+e_j)$ and $(a-e_j+e_i, a+e_i)$ for $j\in \{1,\ldots, d\}\setminus\{i\}$. Since $g\in \ma(\R^d)$, all third-order derivatives of the components of $g$ are at most of order $(1+\|\ve a\|)^{-K}$ in a $2\ve$-neighborhood of $\ve a$, where $K$ is an arbitrary positive constant that will be chosen later. Thus, by Taylor expansion, 
\begin{align*}
&\sum_{e'\in N^+(e)} x(e')  = \sum_{j\ne i} (g_j(\ve a+\ve e_i)+ g_j(\ve a-\ve e_j))\\
&= \sum_{j\ne i} \biggl(2g_j(\ve a) + \ve \partial_i g_j(\ve a) - \ve \partial_j g_j(\ve a) + \frac{\ve^2}{2} \partial_i^2 g_j (\ve a) + \frac{\ve^2}{2} \partial_j^2 g_j(\ve a)\biggr) \\
&\qquad \qquad + O(\ve^3(1+\|\ve a\|)^{-K}).
\end{align*}
Similarly, 
\begin{align*}
\sum_{e'\in N^-(e)} x(e')  &= \sum_{j\ne i} (g_i(\ve a+\ve e_j) + g_i(\ve a-\ve e_j) + g_j(\ve a) + g_j(\ve a-\ve e_j+\ve e_i))\\
&= \sum_{j\ne i} \biggl(2g_i(\ve a) + \ve^2\partial_j^2g_i(\ve a) +2g_j(\ve a) +\ve \partial_i g_j(\ve a) - \ve \partial_j g_j(\ve a) \\
&\qquad \qquad \qquad +\frac{\ve^2}{2} \partial_i^2 g_j(\ve a) + \frac{\ve^2}{2}\partial_j^2 g_j(\ve a) - \ve^2 \partial_i\partial_j g_j(\ve a)\biggr) \\
&\qquad \qquad  + O(\ve^3 (1+\|\ve a\|)^{-K}). 
\end{align*}
Thus, we get
\begin{align*}
&(2d-2) x(e) + \sum_{e'\in N^+(e)} x(e') - \sum_{e'\in N^-(e)} x(e')\\
&= \ve^2\sum_{j\ne i}(\partial_i \partial_j g_j(\ve a) - \partial_j^2 g_i(\ve a) ) +O(\ve^3 (1+\|\ve a\|)^{-K})\\
&= \ve^2\sum_{j=1}^d(\partial_i \partial_j g_j(\ve a) - \partial_j^2 g_i(\ve a) ) +O(\ve^3 (1+\|\ve a\|)^{-K}).
\end{align*}
Putting it all together, we see that for any $e = (a,a+e_i)\in E$ that is not a boundary edge, 
\begin{align}\label{zedef}
z(e) = \ve^d g_i(\ve a) - \ve^d \Delta g_i(\ve a) + \ve^d \sum_{j=1}^d \partial_i\partial_j g_j(\ve a) + O(\ve^{d+1}(1+\|\ve a\|)^{-K}). 
\end{align}
Next, define $f := Q_1 g$, where $Q_1$ is the operator defined in equation \eqref{qdef} with $\lambda =1$. For $e = (a,a+e_i)\in E$, define $w(e) := \ve^d f_i(\ve a)$. Then, if $e$ not a boundary edge, \eqref{zedef} shows that
\begin{align*}
z(e)-w(e) &= O(\ve^{d+1}(1+\|\ve a\|)^{-K}). 
\end{align*}
On the other hand, since $g$ and $f$ are Schwartzian $1$-forms, we have the crude estimate 
\[
z(e) -w(e) = O(\ve^{d-2} (\ve L)^{-K})
\]
for any boundary edge $e$. Thus, considering $z$ and $w$ as vectors in $\R^E$, we get
\begin{align}
\|z-w\|^2 &= O\biggl(\sum_{r=0}^L (1+r)^{d-1} \ve^{2d+2} (1+\ve r)^{-2K} + L^{d-1} \ve^{2d-4}(\ve L)^{-2K} \biggr)\notag \\
&=  O\biggl(\sum_{r=0}^{\lfloor 1/\ve\rfloor} (1+r)^{d-1}\ve^{2d+2} + \sum_{r = \lfloor1/\ve \rfloor}^L r^{-2K} \ve^{2d+2-2K} +\ve^{2d-4-2K} L^{d-1-2K}\biggr) \notag \\
&= O(\ve^{d+2}) + O(\ve^{2d+1}) + O(\ve^{2d-4-2K} L^{d-1-2K}),\label{zwineq}
\end{align}
provided that $K > 1/2$. 
Since $L^{-1}\le \ve^{1+\delta}$, we can choose $K$ large enough (depending on $d$ and $\delta$) such that
\[
\ve^{2d-4-2K} L^{d-1-2K} \le \ve^{2d-4-2K} \ve^{(1+\delta)(2K+1-d)}\le \ve^{d+2}.
\]
Thus, with such a choice of $K$, we get 
\begin{align}\label{zwdifference}
\|z-w\|^2 = O(\ve^{d+2}). 
\end{align}
Let $D$ denote the Voronoi cell of $0$ in $\Z^d$, so that the Voronoi cell of $\ve a\in \ve \Z^d$ is $\ve D + \ve a$. Then note that
\begin{align*}
Y(f) 
&=\sum_{e = (a, a+e_i)\in E}  Y_i(\ve a) \int_{\ve D + \ve a}  f_i(q) dq.
\end{align*}
Thus, $Y(f)$ is a Gaussian random variable with mean zero and variance $u^*R u$, 
where $u\in \R^E$ is defined as
\begin{align}\label{udefinition}
u(a,a+e_i) := \int_{\ve D + \ve a} f_i(q) dq.
\end{align}
Since $f\in \ma(\R^d)$, we have that for any $e = (a,a+e_i)\in E$, 
\begin{align*}
u(e)- \vol(\ve D + \ve a) f_i(\ve a) &= \int_{\ve D + \ve a} (f_i (q)-f_i(a)) dq \\
&= O(\ve (1+\|\ve a\|)^{-K}) \vol(\ve D + \ve a).
\end{align*}
But note that
\[
\vol(\ve D + \ve a) = \vol(\ve D) = \ve^d \vol(D) = \ve^d,
\]
where the last identity holds because $\vol(D)=1$ by a simple density argument. Thus, 
\[
u(e)-w(e) = O(\ve^{d+1}(1+\|\ve a\|)^{-K}).
\] 
From this, by the same calculation as in \eqref{zwineq}, we get
\begin{align}\label{uwrel}
\|u-w\|^2 = O(\ve^{d+2}).
\end{align}
Moreover, note that  for any $e = (a,a+e_i)\in E$, $u(e)$ and $w(e)$ are both of order $\ve^d (1+\|\ve a\|)^{-K}$. Thus, 
\begin{align}
\max\{\|u\|^2, \|w\|^2\} &= O\biggl(\sum_{r=0}^\infty (1+r)^{d-1} \ve^{2d} (1+\ve r)^{-2K}\biggr)\notag \\
&= O\biggl(\sum_{r=0}^{\lfloor 1/\ve\rfloor} (1+r)^{d-1} \ve^{2d}  + \sum_{r=\lfloor 1/\ve\rfloor} r^{d-1} \ve^{2d} (\ve r)^{-2K}\biggr)\notag \\
&= O(\ve^d), \label{uwmax}
\end{align}
choosing $K$ large enough. 
Recalling our earlier observation that $\|R\|=O(\ve^{-d})$ (where $\|R\|$ denotes the $\ell^2$ operator norm of $R$), and combining with \eqref{uwrel} and \eqref{uwmax}, we get
\begin{align*}
|u^*Ru - w^*R w| &\le |u^*R(u-w)| + |(u-w)^*Rw|\\
&\le \|u\|\|R\|\|u-w\| + \|u-w\|\|R\|\|w\|\\
&= O(\ve).
\end{align*}
Similarly, by \eqref{zwdifference},
\[
|w^*Rw - w^*Rz| \le \|w\|\|R\|\|w-z\| = O(\ve). 
\]
Combining the last two displays, and observing that $Rz = RQx = x$, we get
\begin{align}\label{uruwx}
|u^*R u - w^*x| = O(\ve).
\end{align}
But notice that
\begin{align*}
w^*x &= \sum_{e = (a,a+e_i)\in E}\ve^d g_i(\ve a) f_i(\ve a).
\end{align*}
If $\ve \to 0$, $L\to\infty$ and $L\ge \ve^{-1-\delta}$, then from the above formula and the fact that $f,g\in \ma(\R^d)$, it follows that $w^*x$ tends to $(f,g)$. But $f=Q_1 g$, and therefore, by Lemma \ref{qlemma}, $g = R_1 f$. Thus, by \eqref{uruwx} and the fact that $Y(f)\sim \mathcal{N}(0, u^*Ru)$, it follows that $Y(f)$ converges in law to $\mathcal{N}(0, (f, R_1 f))$. Since $Q_1$ is a bijection of $\ma(\R^d)$ and $g$ is an arbitrary element of $\ma(\R^d)$, this conclusion holds for all $f\in \ma(\R^d)$. Thus, by Lemma \ref{existlmm}, the Euclidean Proca field with $\lambda = 1$ exists, and $Y$ converges to it in law. The existence for arbitrary $\lambda$ follows from Lemma \ref{scalinglmm}.
\end{proof}

\subsection{Proof of Lemma \ref{scalinglmm}}\label{scalinglmmproof}
Take any $f\in \ma(\R^d)$, and let $g:= \tau_{a,b} f$. Note that 
\begin{align*}
\Delta g(x) &= a^{-2} \Delta f((x-b)/a) = a^{-2} \tau_{a,b} \Delta f(x). 
\end{align*}
Thus,
\begin{align*}
(-\Delta + \lambda I) \tau_{a,b} f &=  -a^{-2}\tau_{a,b} \Delta f + \lambda \tau_{a,b} f\\
&= a^{-2}\tau_{a,b}(-\Delta + a^2\lambda I) f.
\end{align*}
To put it more succintly,
\[
(-\Delta + \lambda I) \tau_{a,b} = a^{-2}\tau_{a,b}(-\Delta + a^2\lambda I). 
\]
Multiplying on the left by $K_\lambda$ throughout and on the right by $K_{a^2\lambda}$ throughout, and applying Lemma~\ref{klemma},  we get
\begin{align*}
\tau_{a,b} K_{a^2\lambda} &= a^{-2}K_\lambda \tau_{a,b}.
\end{align*}
Similarly, since
\[
\partial_i g(x) = a^{-1}\partial_i f((x-b)/a) = a^{-1}\tau_{a,b} \partial_ i f,
\]
we have $\partial_i \tau_{a,b} = a^{-1}\tau_{a,b}\partial_i$. Thus,
\begin{align}
R_\lambda g &:= \sum_{i=1}^d \biggl(K_\lambda \tau_{a,b}f_i - \lambda^{-1}\sum_{j=1}^d \partial_i \partial_j K_\lambda\tau_{a,b} f_j\biggr) dx_i\notag \\
&= a^2 \sum_{i=1}^d \biggl(\tau_{a,b} K_{a^2\lambda} f_i - \lambda^{-1}\sum_{j=1}^d \partial_i\partial_j \tau_{a,b} K_{a^2\lambda} f_j\biggr) dx_i\notag \\
&= a^2 \tau_{a,b}\sum_{i=1}^d \biggl( K_{a^2\lambda} f_i - a^{-2}\lambda^{-1}\sum_{j=1}^d \partial_i\partial_jK_{a^2\lambda} f_j\biggr) dx_i\notag \\
&= a^2\tau_{a,b} R_{a^2\lambda} f. \label{rtranslate}
\end{align}
Now, note that for any $u,v\in \ma(\R^d)$, 
\begin{align*}
(\tau_{a,b} u, \tau_{a,b} v) &= \sum_{i=1}^d \int_{\R^d} u_i((x-b)/a) v_i((x-b)/a) dx\\
&= a^d \sum_{i=1}^d \int_{\R^d} u_i(z)v_i(z) dz = a^d (u, v). 
\end{align*}
From this and \eqref{rtranslate}, we get
\[
(g, R_\lambda g) = a^2(\tau_{a,b}f, \tau_{a,b} R_{a^2\lambda} f) = a^{d+2}(f, R_{a^2\lambda} f).
\]
Now recall that $Y(f) = a^{(d-2)/2} a^{-d} X(g)$. Therefore, $Y$ is a Gaussian random variable with mean zero and variance 
\[
a^{d-2} a^{-2d} (g, R_\lambda g) = (f, R_{a^2\lambda } f). 
\]
This proves that $Y$ is a Euclidean Proca field with parameter $a^2\lambda$. 

\subsection{Proof of Lemma \ref{masslmm}}\label{masslmmproof}
We need the following lemmas.
\begin{lmm}\label{mass1}
As $\|x\|\to \infty$,
\begin{align*}
K_\lambda(x) \sim \frac{\lambda^{(d-3)/4}e^{-\sqrt{\lambda}\|x\|}}{2(2\pi)^{(d-1)/2}\|x\|^{(d-1)/2}},
\end{align*}
meaning that the ratio of the two sides tends to $1$ as $\|x\|\to \infty$.
\end{lmm}
\begin{proof}
Note that for any $x$ and $t$,
\begin{align*}
\frac{\|x\|^2}{4t} +\lambda t &= \biggl(\frac{\|x\|}{2\sqrt{t}} - \sqrt{\lambda t}\biggr)^2 + \sqrt{\lambda}\|x\|\\
&= \frac{(\|x\|-2\sqrt{\lambda} t)^2}{4t} + \sqrt{\lambda } \|x\|.
\end{align*}
Thus,
\begin{align*}
K_\lambda(x) &= \int_0^\infty \frac{1}{(4\pi t)^{d/2}} \exp\biggl(-\frac{\|x\|^2}{4t} - \lambda t\biggr) dt\\
&= e^{-\sqrt{\lambda} \|x\|} \int_0^\infty \frac{1}{(4\pi t)^{d/2}} \exp\biggl(- \frac{(\|x\|-2\sqrt{\lambda} t)^2}{4t}\biggr) dt.
\end{align*}
Making the change of variable
\[
s = \frac{1}{\sqrt{\|x\|}} \biggl(t - \frac{\|x\|}{2\sqrt{\lambda}}\biggr)
\]
in the above integral yields
\begin{align*}
&K_\lambda(x) = \frac{e^{-\sqrt{\lambda} \|x\|}}{(4\pi)^{d/2}} \int_{-\frac{1}{2}\sqrt{\frac{\|x\|}{\lambda}}}^\infty \biggl( s\sqrt{\|x\|} + \frac{\|x\|}{2\sqrt{\lambda}}\biggr)^{-d/2} \exp\biggl(-\frac{4\lambda^{3/2}\|x\|s^2}{4s\sqrt{\lambda \|x\|} + 2\|x\|}\biggr)\sqrt{\|x\|} ds\\
&= \frac{e^{-\sqrt{\lambda} \|x\|}}{(4\pi)^{d/2}\|x\|^{(d-1)/2}}\int_{-\frac{1}{2}\sqrt{\frac{\|x\|}{\lambda}}}^\infty \biggl(\frac{ s}{\sqrt{\|x\|}} + \frac{1}{2\sqrt{\lambda}}\biggr)^{-d/2}\exp\biggl(-\frac{4\lambda^{3/2}\|x\|s^2}{4s\sqrt{\lambda \|x\|} + 2\|x\|}\biggr)ds.
\end{align*}
By the dominated convergence theorem, it is easy to show that the above integral converges to
\[
\int_{-\infty}^\infty (2\sqrt{\lambda})^{d/2}e^{-2\lambda^{3/2} s^2}ds = (2\sqrt{\lambda})^{d/2} \sqrt{\frac{2\pi}{4\lambda^{3/2}}}
\]
as $\|x\|\to \infty$. This gives the desired result.
\end{proof}
\begin{cor}\label{masscor}
Take any $y\in \R^d$. Let $x_n$ be a sequence of points in $\R^d$ such that $\|x_n\|\to \infty$ and $\|x_n\|^{-1} x_n \to u\in S^{d-1}$. Then 
\[
\lim_{n\to \infty} \frac{K_\lambda(x_n+y)}{K_\lambda(x_n)} = e^{-\sqrt{\lambda} u\cdot y}. 
\]
\end{cor}
\begin{proof}
Note that as $\|x\|\to \infty$ (with $y$ fixed), 
\begin{align*}
\|x+y\| - \|x\| &= \sqrt{\|x\|^2+\|y\|^2 + 2x\cdot y} - \|x\|\\
&= \|x\| \sqrt{ 1+ \frac{\|y\|^2}{\|x\|^2} + \frac{2x\cdot y}{\|x\|^2}} - \|x\|\\
&= \|x\|\biggl(1+  \frac{x\cdot y}{\|x\|^2} + O(\|x\|^{-2})\biggr) - \|x\|\\
&= \frac{1}{\|x\|} x\cdot y + O(\|x\|^{-1}).
\end{align*}
By Lemma \ref{mass1}, this proves the claim.
\end{proof}
\begin{lmm}\label{mass2}
For any two bump functions $f,g\in \ms(\R^d)$ and any sequence $x_n$ such that $\|x_n\|\to \infty$ and $\|x_n\|^{-1} x_n \to u\in S^{d-1}$, 
\[
\lim_{n\to \infty} \frac{\int f(y) K_\lambda g^{x_n}(y) dy}{K_\lambda(x_n)} = \iint f(y)g(y-v) e^{-\sqrt{\lambda}u\cdot v} dv dy = \Psi(f,g,u,\lambda).
\]
\end{lmm}
\begin{proof}
Note that for any $x$,
\begin{align}\label{masseq1}
\int f(y) K_\lambda g^x(y) dy &= \iint f(y) K_\lambda(z) g^x(y-z) dz dy\notag \\
&=  \iint f(y) K_\lambda(z) g(x+y-z) dz dy\notag \\
&=  \iint f(y) K_\lambda(v+x) g(y-v) dv dy.
\end{align}
By Corollary \ref{masscor}, for any fixed $v$,
\begin{align}\label{masseq2}
\lim_{n\to \infty}\frac{K_\lambda(v+x_n)}{K_\lambda(x_n)} = e^{-\sqrt{\lambda} u\cdot v}. 
\end{align}
Since $f$ and $g$ are bump functions, there is some $C$ such that if $\|y\|>C$ or $\|y-v\|> C$, then $f(y)g(y-v)=0$. But if both of these quantities are $\le C$, then $\|v\|\le 2C$. Take any such $v$. Now, Lemma \ref{mass1} implies that there are positive constants $B$, $C_1$ and $C_2$ such that for any $y$ with $\|y\|\ge B/2$, we have
\begin{align}\label{kybound}
C_1 \le \frac{K_\lambda(y)}{\|y\|^{-(d-1)/2} e^{-\sqrt{\lambda}\|y\|}}\le C_2.
\end{align}
Without loss, we may assume that $B\ge 4C$. Then, for any $n$ is so large that $\|x_n\|\ge B$, we have $\|v+x_n\|\ge \|x_n\|-\|v\|\ge B - 2C\ge B/2$. Thus, \eqref{kybound} holds for $y=x_n$ and $y=x_n+v$. By the calculation in the proof of Corollary \ref{masscor}, this implies that there are positive constants $C_3$ and $C_4$ such that for any $n$ satisfying $\|x_n\|\ge B$ and any $v$ such that $\|v\|\le 2C$, we have
\begin{align}\label{masseq3}
C_3 \le \frac{K_\lambda(v+x_n)}{K_\lambda(x_n)} \le C_4. 
\end{align}
By \eqref{masseq1}, \eqref{masseq2}, \eqref{masseq3} and the dominated convergence theorem, we get the desired result.
\end{proof}

We are now ready to complete the proof of Lemma \ref{masslmm}. Take any two bump functions $f$ and $g$ in $\ma(\R^d)$, and let $X$ be a Euclidean Proca field with parameter $\lambda$. Then note that for any $x$,
\begin{align*}
\var(X(f+g^x)) &= (f+g^x, R_\lambda(f+g^x)) \\
&= (f, R_\lambda f) + (g^x, R_\lambda g^x) + (f, R_\lambda g^x) + (g^x, R_\lambda f).
\end{align*}
On the other hand, since $X(f)$ and $X(f^x)$ have mean zero,
\begin{align*}
\var(X(f+g^x)) &= \var(X(f)+X(g^x))\\
&= \var(X(f)) + \var(X(g^x)) + 2\E(X(f)X(g^x))\\
&= (f, R_\lambda f) + (g^x, R_\lambda g^x) + 2\E(X(f)X(g^x)).
\end{align*}
Thus, 
\begin{align*}
\E(X(f)X(g^x)) &= \frac{1}{2}((f, R_\lambda g^x) + (g^x, R_\lambda f)) = (f,R_\lambda g^x),
\end{align*}
where the last identity holds because, by a simple verification, $R_\lambda$ is a self-adjoint operator for the $L^2$ inner product on $\ma(\R^d)$. Now note that
\begin{align*}
(f, R_\lambda g^x) &= \sum_{i=1}^d \iint \biggl(f_i(y)K_\lambda g_i^x(y) - \lambda^{-1}\sum_{j=1}^d f_i(y) \partial_i \partial_j K_\lambda g_j^x(y)\biggr) dy\\
&=  \sum_{i=1}^d \iint \biggl(f_i(y)K_\lambda g_i^x(y) + \lambda^{-1}\sum_{j=1}^d \partial_i f_i(y) K_\lambda (\partial_j g_j)^x(y)\biggr) dy,
\end{align*}
where the second identity follows by integration by parts and the fact that $K_\lambda$, derivatives and translations commute. By Lemma \ref{mass1} and Lemma \ref{mass2}, this completes the proof.

\subsection{Some properties of the discrete Proca field}
In this subsection, we {\it do not} identify opposite faces of $\Lambda$. Recall the discrete Proca field defined in Definition \ref{dmfdef} of Subsection \ref{mmffproof}. In this subsection, we will prove some properties of this field that will be useful later. Define a graph structure on $E$ by declaring that two edges are neighbors if they share a common plaquette. Let $d(e,e')$ denote the graph distance between two edges $e$ and $e'$ on this graph.

\begin{lmm}\label{corrlmm}
Let $X$ be a discrete Proca field on $\Lambda = \{-L,\ldots,L\}^d$ with parameter $\ve^2 \le 1$. 
Then for any $e,e'\in E$,  
\[
|\E(X(e)X(e'))|\le  \ve^{-2} (1-(16d)^{-1}\ve^2)^{d(e,e')}.
\]
\end{lmm}
\begin{proof}
Let $R$ denote the covariance matrix of $\ve^{-(d-2)/2}X$, and recall that $R$ is the inverse of the matrix $Q$ defined in \eqref{qdefinition}. Notice that the sum of the absolute values of the entries on any row of $Q$ is $\le 16(d-1)\ve^{d-2} + \ve^d\le 16d \ve^{d-2}$. By a standard fact about real symmetric matrices, this implies that largest eigenvalue of $Q$ is $\le 16d \ve^{d-2}$. On the other hand, as observed in the proof of Theorem \ref{convthm}, the smallest eigenvalue of $Q$ is $\ge \ve^d$. Thus, if we define
\begin{align*}
S := I - (16d \ve^{d-2})^{-1} Q,
\end{align*}
then $S$ is a real symmetric matrix whose eigenvalues lie between $0$ and $1-(16d)^{-1}\ve^2$. In particular, $\|S\|\le 1-(16d)^{-1}\ve^2 < 1$, where $\|S\|$ denotes the $\ell^2$ operator norm of $S$. Consequently, we can expand
\[
(I-S)^{-1} = \sum_{k=0}^\infty S^k. 
\]
But $(I-S)^{-1} = 16d \ve^{d-2} R$. Thus,
\[
R = (16d\ve^{d-2})^{-1}\sum_{k=0}^\infty S^k. 
\]
Now, note that $S(e,e') =0$ unless $e$ and $e'$ are either the same edge, or share a plaquette. Thus, $S^k(e,e') = 0$ unless $d(e,e')\le k$. On the other hand, for any $e$ and $e'$,
\[
|S^k(e,e')| \le \|S\|^k \le (1-(16d)^{-1}\ve^2)^k.
\]
Combining, we get that for any $e$ and $e'$,
\begin{align*}
|R(e,e')| &\le (16d\ve^{d-2})^{-1}\sum_{k=d(e,e')}^\infty  (1-(16d)^{-1}\ve^2)^k \\
&= (16d\ve^{d-2})^{-1}\frac{(1-(16d)^{-1}\ve^2)^{d(e,e')}}{(16d)^{-1}\ve^2} = \ve^{-d} (1-(16d)^{-1}\ve^2)^{d(e,e')}. 
\end{align*}
This completes the proof of the lemma.
\end{proof}

Let $\partial E$ denote the set of boundary edges of $\Lambda$. Given a vector $x=(x_e)_{e\in \partial E}\in \R^{\partial E}$, let $X'$ be generated from the discrete Proca field on $\Lambda$ with parameter $\ve^2$, conditional on the boundary values being equal to components of $x$. Define a random $1$-form $Y'$ using $X'$ just as we defined $Y$ using $X$ in equation \eqref{yfromx2}. In the following, for an edge $e = (a,a+e_i)$, we will write $Y(e)$ and $Y'(e)$ as shorthands for $Y_i(\ve a)$ and $Y_i'(\ve a)$. 
\begin{lmm}\label{normallmm}
Let $Y'$ be as above and suppose that $L\ge \ve^{-3}\ge 1$. Then for any $f\in \ma(\R^d)$, and any choice of $B>0$, there exists $C$ depending only on $f$, $d$ and $B$ (and not on $L$ or $\ve$), such that $|\E(Y'(f)) - \E(Y(f))| \le CL^{-B} \|x\|$  and $|\var(Y'(f)) - \var(Y(f))| \le CL^{-B}$. 
\end{lmm}
\begin{proof} 
Let $E^\circ:=E \setminus \partial E$. Let $R$ be defined as in \eqref{rdefinition} and $u$ be defined as in \eqref{udefinition}, so that $\var(Y(f))=u^*Ru$, as observed in the proof of Theorem \ref{convthm}. Also, it is clear that $\E(Y(f))=0$. Let $S := (R(e,e'))_{e,e'\in \partial E}$, and for any $e \in E$, let $q(e) := (R(e,e'))_{e'\in \partial E}$. Note that by Lemma \ref{corrlmm}, if an edge $e$ is at an $\ell^\infty$ distance $r$ from the boundary, then
\begin{align}\label{qbbound}
\|q(e)\|^2 &\le |\partial E| \ve^{-2d} e^{-C\ve^2 r} \le C_1L^{d-1} \ve^{-2d} e^{-C_2 \ve^2 r},
\end{align}
where $C$, $C_1$ and $C_2$ depend only on $d$. 
Since $f\in \ma(\R^d)$, for an edge $e$ at $\ell^\infty$ distance $r$ from the origin,
\begin{align}\label{uebound}
|u(e)| \le C\ve^d (1+\ve r)^{-A}
\end{align}
for any choice of $A>0$, where $C$ depends only on $f$, $d$ and $A$. In the following, $C,C_1,C_2,\ldots$ will denote positive constants whose values depend only on $d$, $f$, and the choice of $A$, and may change from line to line. First, note that by \eqref{qbbound} and \eqref{uebound},
\begin{align}\label{sum1}
\sum_{e\in E^\circ} |u(e)|\|q(e)\| &\le  C_1L^{(d-1)/2} \sum_{r=0}^L e^{-C_2 \ve^2 (L-r)}(1+\ve r)^{-A}\notag\\
&\le C_1L^{(d-1)/2} \biggl(\sum_{r\le L/2} e^{-C_2 \ve^2 L/2} + \sum_{L/2< r\le L} (\ve L/2)^{-A}\biggr) \notag\\
&\le C_3 L^{(d+1)/2} (e^{-C_4 \ve^2 L} + \ve^{-A} L^{-A}). 
\end{align}
Similarly,
\begin{align}\label{sum2}
\sum_{e\in \partial E} |u(e)| &\le C_1 L^{d-1 - A}\ve^{d-A}  
\end{align}
and
\begin{align}\label{sum3}
\sum_{e\in E^\circ} |u(e)| &\le C_1\ve^d  \sum_{r=0}^L (1+r)^{d-1}(1+\ve r)^{-A}\notag\\
&\le  C_1\ve^d  \biggl(\sum_{0\le r\le 1/\ve} (1+r)^{d-1} + \sum_{ 1/\ve\le r\le L} r^{d-1}(\ve r)^{-A} \biggr) \notag\\
&\le C_1\ve^d \ve^{-d} + C_2 \ve^d\ve^{-A} \ve^{A-d} \le C_3, 
\end{align}
provided that $A>d$. 
Now, by standard facts about multivariate Gaussian distributions, we have that for $e,e'\in E^\circ$, 
\begin{align}\label{expvar}
\E(Y'(e)) = q(e)^* S^{-1} x, \ \ \ \cov(Y'(e), Y'(e')) = \cov(Y(e), Y(e')) - q(e)^*S^{-1}q(e').
\end{align}
On the other hand, if $e\in \partial E$, then $Y'(e)=x(e)$, and if either of $e$ and $e'$ is in $\partial E$, then $\cov(Y(e), Y(e'))=0$. 
Thus,
\begin{align}\label{ebound}
|\E(Y(f)) - \E(Y'(f))| &\le \sum_{e\in E} |\E(Y'(e)) u(e)| \notag\\
&= \sum_{e\in E^\circ} |u(e) q(e)^*S^{-1} x| + \sum_{e\in \partial E} |x(e) u(e)|,
\end{align}
and 
\begin{align}\label{vbound}
|\var(Y(f)) - \var(Y'(f))| &\le \sum_{e,e'\in \partial E} |\cov(Y(e),Y(e')) u(e) u(e')| \notag \\
&\qquad + 2\sum_{e\in \partial E, \, e'\in E^\circ}  |\cov(Y(e),Y(e')) u(e) u(e')|\notag \\
&\qquad + \sum_{e,e'\in E^\circ} |u(e)u(e') q(e)^*S^{-1} q(e')|.
\end{align}
Since $S$ is a principal submatrix of the real symmetric matrix $R$, the smallest eigenvalue of $S$ is bounded below by the smallest eigenvalue of $R$. Thus, the largest eigenvalue of $S^{-1}$ is bounded above by the largest eigenvalue of $Q=R^{-1}$, which is at most of order $\ve^{d-2}$ (since that is the order of the  maximum of the row sums of the absolute values of the entries of $Q$). Since $S$ is a positive semidefinite matrix, this shows that the $\ell^2$ operator norm of $S^{-1}$ is at most $O(\ve^{d-2})$. By \eqref{sum1}, \eqref{sum2} and \eqref{ebound}, this proves that 
\begin{align*}
|\E(Y(f)) - \E(Y'(f))| &\le\ve^{d-2} \|x\| \sum_{e\in E^\circ} |u(e)| \|q(e)\| + \|x\| \sum_{e\in \partial E} |u(e)| \\
&\le C_1 L^{(d+1)/2} \ve^{d-2}\|x\| (e^{-C_2\ve^2 L} +\ve^{-A}L^{-A}) + C_3 \|x\| L^{d-1-A}\ve^{d-A}.
\end{align*}
If $L\ge \ve^{-3}$, then with a suitably large choice of $A$ (depending only on $f$, $d$ and $B$), the right side above can be made smaller than $CL^{-B}\|x\|$, where $C$, too, depends only $f$, $d$ and $B$. Similarly, by \eqref{uebound}, \eqref{sum1}, \eqref{vbound} and Lemma \ref{corrlmm},
\begin{align*}
&|\var(Y(f)) - \var(Y'(f))| \\
&\le 3 \sum_{e\in \partial E, \, e'\in E} |u(e)u(e') \cov(Y(e), Y(e'))| + \ve^{d-2}\biggl(\sum_{e\in E^\circ} |u(e)|\|q(e)\|\biggr)^2\\
&\le C_1 L^{d-1} \ve^d \sum_{r=0}^L (1+r)^{d-1} (1+\ve L)^{-A} (1+\ve r)^{-A} e^{-C_2\ve^2(L-r)}\\
&\qquad \qquad + C_3 L^{d+1} \ve^{d-2}(e^{-C_4 \ve^2 L} + \ve^{-2A} L^{-2A})\\
&\le C_1 L^{d-1}\ve^{d}(L^d e^{-C_2\ve^2 L} + L^{d-2A} \ve^{-2A}) + C_3 L^{d+1} \ve^{d-2}(e^{-C_4 \ve^2 L} + \ve^{-2A} L^{-2A}).
\end{align*}
Again, if $L\ge \ve^{-3}$, then with a large enough choice of $A$ (depending only on $f$, $d$ and $B$), this bound can be made smaller than $CL^{-B}$. 
\end{proof}

\section{Proofs of the main results}
In this section, we prove the main results --- that is, Theorem \ref{uonethm} and Theorem \ref{sutwothm}. The proofs are quite similar, so we first work out the $\sutwo$ case in complete detail, and then the $\uone$ case is worked out by analogous arguments where the complete details are sometimes omitted.
\subsection{Probability density of stereographic projection}\label{stereosec}
Consider the following map from $\R^n$ to the unit sphere $S^n$ in $\R^{n+1}$. For $x\in \R^n$, let $P_n(x)$ denote the unique point on the line joining $(1,x)\in \R^{n+1}$ to the point $-e_1 =(-1,0,\ldots,0)\in S^{n}$.  It is easy to check by explicit calculation that
\[
P_n(x) = \biggl(\frac{4-\|x\|^2}{4+\|x\|^2}, \frac{4x_1}{4+\|x\|^2}, \ldots,\frac{4x_n}{4+\|x\|^2}\biggr),
\]
because $\|P_n(x)\|=1$ and $P_n(x)$ is a convex combination of $(1,x)$ and $-e_1$. Note that $P_n$ is simply the inverse of the stereographic projection map $\sigma_n$ defined in Subsection \ref{scale1} --- see Figure~\ref{stereo}.
\begin{lmm}\label{stereolmm}
Consider the probability measure $\mu$ on $\R^n$ that has probability density (with respect to Lebesgue measure) proportional to $(4+\|x\|^2)^{-n}$. The pushforward of this measure under $P_n$ is the uniform probability measure on $S^n$. 
\end{lmm}
\begin{proof}
Let $\nu$ be the pushforward of $\mu$ under $P_n$. It is easy to see that $\nu$ is preserved under rotations that fix the first axis in $\R^{n+1}$, because $\mu$ is invariant under all rotations in $\R^n$. Thus, to show that $\nu$ is the uniform probability measure on $S^n$, it suffices to prove that for any $t\in [-1,1]$, the $\nu$-measure of the cap $C_t := \{x\in S^n: x_1\le t\}$ is proportional to the $n$-dimensional volume of the cap. But 
\begin{align}\label{pnct}
P_n^{-1}(C_t) &= \biggl\{x\in \R^n: \frac{4-\|x\|^2}{4+\|x\|^2} \le t\biggr\}\notag \\
&= \biggl\{x\in \R^n: \|x\|^2 \ge \frac{4(1-t)}{1+t}\biggr\}. 
\end{align}
Thus, 
\begin{align*}
\nu(C_t) &= \mu(P_n^{-1}(C_t))\\
&\propto \int_{\{x\in \R^n\, :\, \|x\|^2 \ge 4(1-t)/(1+t)\}} \frac{1}{(4+\|x\|^2)^n} dx\\
&\propto \int^\infty_{\sqrt{4(1-t)/(1+t)}} \frac{r^{n-1}}{(4+r^2)^n} dr,
\end{align*}
where the last step was obtained by changing to polar coordinates. Making the change of variable $u=(4-r^2)/(4+r^2)$ in the above integral, we get
\[
\nu(C_t) \propto \int_{-1}^t (1-u^2)^{(n-2)/2} du.
\]
Thus, we have to show that if a point $X$ is picked uniformly at random from $S^n$, then the first coordinate of $X$ has probability density proportional to $(1-u^2)^{(n-2)/2}$ in $[-1,1]$. This is a standard fact, easily provable by geometric considerations or otherwise. 
\end{proof}
Lemma \ref{stereolmm} implies the following corollary.
\begin{cor}\label{densitycor}
Let $n$ and $N$ be two positive integers, and let $f:(S^n)^N \to \R$ be any bounded measurable function. There is a constant $C_n$ depending only on $n$ such that 
\begin{align*}
&\int_{(S^n)^N} f(x_1,\ldots,x_N) dx_1\cdots dx_N \\
&= C_n^N\int_{(\R^n)^N} f(P_n(y_1),\ldots, P_n(y_N))\prod_{i=1}^N \frac{1}{(4+\|y_i\|^2)^n}dy_1\cdots dy_N,
\end{align*}
where $dx_i$ on the left denotes integration with respect to the uniform probability measure on $S^n$ and $dy_i$ on the right indicates integration with respect to Lebesgue measure on $\R^n$.
\end{cor}

\subsection{Probability density after gauge fixing}
Recall the field $V$ obtained from $\uone$ theory after unitary gauge fixing in Subsection~\ref{uonesec}, for the theory on the finite box $\Lambda$ (and not in infinite volume). The following lemma gives the probability density of $V$.
\begin{lmm}\label{vdens1}
The probability density of $V$ with respect to the product of Haar probability measures on $\uone^E$ is proportional to
\begin{align*}
\exp\biggl(\frac{1}{g^2}\sum_{p\in P} \Re(U_p) + \alpha^2\sum_{e\in E} \Re(U_e)\biggr)
\end{align*}
at $U\in \uone^E$.
\end{lmm}
\begin{proof}
Take any bounded measurable function $f:\uone^E \to \R$. Let $g$ denote the map that takes $(U, \phi)$ to $V$. That is, if $W:= g(U,\phi)$, then for any $e=(x,y)\in E$, $W_e = \phi_x^* U_e \phi_y$. Note that
\begin{align*}
\E(f(V)) &= \frac{\iint f(g(U,\phi)) e^{S(U,\phi)} dU d\phi}{\iint e^{S(U,\phi)} dU d\phi},
\end{align*}
where $S$ is the Yang--Mills action defined in equation \eqref{action1} and $dU$ and $d\phi$ stand for integration with respect to the product Haar measure on $\uone^E$ and the product uniform measure on $(S^1)^\Lambda$, respectively.

Now take any $\phi$, and consider the map $U\mapsto g(U,\phi)$. It is easy to see that this map preserves the product Haar measure. Also, it is easy to see that for any $U$,
\[
S(U,\phi) = S(g(U,\phi), 1),
\]
where $1$ denotes the Higgs field that is equal to $1$ everywhere. Thus, for any $\phi$, 
\begin{align*}
\int f(g(U,\phi)) e^{S(U,\phi)} dU &= \int f(g(U,\phi)) e^{S(g(U,\phi),1)} dU = \int f(U) e^{S(U,1)} dU. 
\end{align*}
Similarly,
\[
\int e^{S(U,\phi)} dU = \int e^{S(U,1)} dU.
\]
Combining, we get 
\begin{align*}
\E(f(V)) &= \frac{\iint f(U) e^{S(U,1)} dU d\phi}{\iint e^{S(U,1)} dU d\phi}= \frac{\int f(U) e^{S(U,1)} dU}{\int e^{S(U,1)} dU}.
\end{align*}
Thus, the probability density of $V$ is proportional to $e^{S(U,1)}$ at $U$, which is what we wanted to show.
\end{proof}
Next, recall the field $V$ obtained from $\sutwo$ theory after unitary gauge fixing in Subsection \ref{sutwosec}, again on the finite box $\Lambda$. The following lemma gives the probability density of $V$. This is very similar to Lemma \ref{vdens1}, but with a crucial difference in one of the constants --- there is now a  denominator $2$ below $\alpha^2$. 
\begin{lmm}\label{vdens2}
The probability density of $V$ with respect to the product of Haar probability measures on $\sutwo^E$ is proportional to
\begin{align*}
\exp\biggl(\frac{1}{g^2}\sum_{p\in P} \Re(\tr(U_p)) + \frac{\alpha^2}{2}\sum_{e\in E} \Re(\tr(U_e))\biggr)
\end{align*}
at $U\in \sutwo^E$.
\end{lmm}
\begin{proof}
Take any bounded measurable function $f:\sutwo^E \to \R$. Let $g$ denote the map that takes $(U, \phi)$ to $V$. Recall that in this case, $g$ is defined as follows. Given $\phi$, for each $x$ we let $\theta_x$ be the unique element of $\sutwo$ such that $\theta_x \phi_x = e_1$ (uniqueness follows from the fact that there is a unique element of $\sutwo$ which sends $e_1$ to $\phi_x$). Then we define $W:= g(U,\phi)$ as $W_e = \theta_x U_e \theta_y^*$ for every $e=(x,y)\in E$. Note that
\begin{align*}
\E(f(V)) &= \frac{\iint f(g(U,\phi)) e^{S(U,\phi)} dU d\phi}{\iint e^{S(U,\phi)} dU d\phi},
\end{align*}
where $S$ is the Yang--Mills action defined in equation \eqref{action2} and $dU$ and $d\phi$ stand for integration with respect to the product Haar measure on $\sutwo^E$ and the product uniform measure on $(S^3)^\Lambda$, respectively.

Now take any $\phi$, and consider the map $U\mapsto g(U,\phi)$. As before, this map preserves the product Haar measure. Now note that for any $U$ and $W:=g(U,\phi)$, 
\begin{align*}
S(U,\phi) &= \frac{1}{g^2}\sum_{p\in P} \Re(\tr(U_p)) + \alpha^2 \sum_{e=(x,y)\in E} \Re(\phi_x^*U_e \phi_y)\\
&= \frac{1}{g^2}\sum_{p\in P} \Re(\tr(W_p)) + \alpha^2 \sum_{e=(x,y)\in E} \Re(\phi_x^*\theta_x^* \theta_xU_e \theta_y^* \theta_y\phi_y)\\
&= \frac{1}{g^2}\sum_{p\in P} \Re(\tr(W_p)) + \alpha^2 \sum_{e=(x,y)\in E} \Re(e_1^* W_e e_1).
\end{align*}
But for any $A\in \sutwo$, 
\[
\Re(e_1^* Ae_1) =\frac{1}{2}\Re(\tr(A)).
\]
Thus, we get
\[
S(U,\phi) = \frac{1}{g^2}\sum_{p\in P} \Re(\tr(W_p)) + \frac{\alpha^2}{2} \sum_{e=(x,y)\in E} \Re(\tr(W_e)) =:  H(W).
\]
This shows that for any $\phi$, 
\begin{align*}
\int f(g(U,\phi)) e^{S(U,\phi)} dU &= \int f(g(U,\phi)) e^{H(g(U,\phi))} dU = \int f(U) e^{H(U)} dU. 
\end{align*}
Similarly,
\[
\int e^{S(U,\phi)} dU = \int e^{H(U)} dU.
\]
As before, this completes the proof.
\end{proof}

\subsection{The key estimate for $\sutwo$ theory}
Let $V$ be the field obtained from $\sutwo$ theory after unitary gauge fixing in Subsection~\ref{sutwosec}, the finite box $\Lambda$. Let $\|\cdot\|$ denote the Euclidean (i.e., Frobenius) norm on $\sutwo$. Note that for any $U\in \sutwo$, 
\begin{align}\label{iunorm}
\|I-U\|^2 &= \tr((I-U)^*(I-U))\notag \\
&= \tr(I) + \tr(U^*U) - \tr(U^*) - \tr(U)\notag \\
&= 4 - 2\Re(\tr(U)). 
\end{align}
Together with Lemma \ref{vdens2}, this shows that  probability density of $V$ is proportional to $e^{-H(U)}$, where 
\[
H(U) := \frac{1}{2g^2}\sum_{p\in P} \|I-U_p\|^2 +\frac{\alpha^2}{4} \sum_{e\in E} \|I-U_e\|^2.
\]
The following lemma gives the key estimate for proving Theorem \ref{sutwothm}.
\begin{lmm}\label{keylmm}
Suppose that $\alpha \ge 2$ and $\alpha g\le 1$. Then for any $e\in E$, 
\[
\E\|I - V_e\|^2 \le \frac{C}{\alpha^4 g^2} + \frac{C\log \alpha}{\alpha^2},
\]
where $C$ depends only on the dimension $d$. 
\end{lmm}
\begin{proof}
Throughout this proof, $C, C_1,C_2,\ldots$ will denote positive constants that only depend on $d$, whose values may change from line to line. Let $\Sigma'$ be the subset of $\Sigma = \sutwo^E$ consisting of all $U$ such that $\|I-U_e \|\le \alpha^{-1}$ for all $e$.  Using the correspondence $\tau$ between $\sutwo$ and $S^3$ defined in Subsection \ref{scale2}, it is easy to see that the normalized Haar measure of $\Sigma'$ is $\ge (C_1\alpha)^{-C_2L^d}$. Also, it is easy to see that for any $U\in \Sigma'$ and $p\in P$, $\|I-U_p\|\le 4\alpha^{-1}$. Thus, if $p(U) := e^{-H(U)}$, then for $U\in \Sigma'$, we have  
\[
p(U) \ge e^{-CL^d(\alpha g)^{-2}}.
\]
This shows that 
\[
\int p(U) \prod_{e\in E} dU_e \ge e^{-C_1L^d(\alpha g)^{-2}} (C_2\alpha)^{-C_3 L^d}. 
\]
Consequently, 
\begin{align*}
\frac{1}{\int p(U) \prod_{e\in E} dU_e} \le e^{C_1L^d(\alpha g)^{-2}} (C_2\alpha)^{C_3 L^d}. 
\end{align*}
But
\begin{align*}
\frac{1}{\int p(U) \prod_{e\in E} dU_e} = \frac{\int p(U)^{-1} p(U) \prod_{e\in E} dU_e}{\int p(U) \prod_{e\in E} dU_e} =  \E(p(V)^{-1}). 
\end{align*}
Thus, for any $t\ge 0$,
\begin{align*}
\P(H(V)\ge t) &= \P(p(V)^{-1}\ge e^t)  \le e^{-t} \E(p(V)^{-1})\\ 
&\le e^{-t}e^{C_1L^d(\alpha g)^{-2}} (C_2\alpha)^{C_3 L^d}. 
\end{align*}
But also, $\P(H(V)\ge t) \le 1$. Thus, 
\begin{align*}
\E(H(V)) &=\int_0^\infty \P(H(V)\ge t) dt\\
&\le \int_0^\infty \min\{e^{-t}e^{C_1L^d(\alpha g)^{-2}} (C_2\alpha)^{3 L^d}, 1\} dt\\
&\le \frac{C_3L^d}{\alpha^2g^2} + C_4L^d \log \alpha.
\end{align*}
But $H(V) \ge \frac{\alpha^2}{4} \sum_{e\in E} \|I- V_e \|^2$, and by symmetry (due to periodic boundary), $\E\|I - V_e\|^2$ is the same for all $e$. This completes the proof.
\end{proof}

\subsection{Local free field approximation for $\sutwo$ theory}
In this subsection, we will show that the field $A$ defined in Subsection \ref{scale2} behaves approximately like a discrete Proca field in a small enough neighborhood of the origin. In the following, we will use the following convention. We will work with the field defined on the finite box $\Lambda$ for most of the proof, only taking $L\to\infty$ at the very end. Thus, in the following discussion $A = (A_e)_{e\in E}$ is defined only on $E$.

Given a configuration $B\in (\R^3)^E$ and an edge $e\in E$, we will denote by $O(g^a B_e^b)$ any quantity whose absolute value is bounded above by $Cg^a (1+\|B_e\|)^b$ for some constant $C$ depending only on $d$ (not to be confused with the component $B_e^b$ of $B_e$). Also, for any plaquette $p$ bounded by four edges $e_1,e_2,e_3,e_4\in E$, we will denote by $O(g^a B_p^b)$ any quantity whose absolute value is bounded above by $C g^a \max_{1\le i\le 4}(1+\|B_{e_i}\|)^b$, where $C$ depends only on $d$. A matrix whose terms are all  $O(g^a B_p^b)$ will also be denoted by $O(g^a B_p^b)$. More generally, for any expression $X$, $O(X)$ will denote  any quantity whose absolute value is bounded above by $C|X|$, where $C$ depends only on $d$. Lastly, if the edges $e_1,e_2,e_3,e_4$ are numbered such that the left endpoint of $e_1$ is the smallest vertex of $p$ in the lexicographic ordering, $e_4$ is incident to the left endpoint of $e_1$, and $e_2$ is incident to the right endpoint of $e_1$ (see Figure \ref{plaquette2}), then we define
\[
B_p := B_{e_1} + B_{e_2} - B_{e_3} - B_{e_4}. 
\]
The following lemma gives a useful representation of the probability density function of~$A$.
\begin{lmm}\label{approxdens}
The probability density function of the random field $A$, at a point $B\in (\R^3)^E$, is proportional to
\begin{align*}
\exp\biggl(-\frac{1}{2} \sum_{p\in P}(\|B_p\|^2 + O(g^2 B^{16}_p))- \frac{\alpha^2g^2}{4}\sum_{e\in E}(\|B_e\|^2 + O(g^2 B^4_e))\biggr). 
\end{align*}
\end{lmm}
\begin{proof}
Let $\tau$ be the map defined in Subsection \ref{scale2} that takes $\sutwo$ to $S^3$, and let $P_3:\R^3 \to S^3$ be the map defined in Subsection \ref{stereosec}. Let $\xi := \sqrt{2} g^{-1} P_3^{-1}\circ \tau$, so that $A_e = \xi(V_e)$. By Lemma~\ref{vdens2}, Corollary \ref{densitycor}, and the identity \eqref{iunorm}, the probability density of $A$ with respect to Lebesgue measure on $(\R^3)^E$ at a point $B\in (\R^3)^E$ is proportional to
\begin{align}\label{newdensity}
\exp\biggl(-\frac{1}{2g^2}\sum_{p\in P} \|I - \xi^{-1}(B)_p\|^2 - \frac{\alpha^2}{4} \sum_{e\in E} \|I-\xi^{-1}(B)_e\|^2\biggr)\prod_{e\in E} \frac{1}{(4 + \frac{1}{2}g^2\|B_e\|^2)^3}. 
\end{align}
Now, note that
\begin{align*}
P_3(2^{-1/2}g B_e) &= \biggl(\frac{1 - \frac{1}{8}g^2 \|B_e\|^2}{1+\frac{1}{8}g^2\|B_e\|^2}, \frac{\frac{1}{\sqrt{2}}gB_e^1}{1+\frac{1}{8}g^2\|B_e\|^2}, \frac{\frac{1}{\sqrt{2}}gB_e^2}{1+\frac{1}{8}g^2\|B_e\|^2}, \frac{\frac{1}{\sqrt{2}}gB_e^3}{1+\frac{1}{8}g^2\|B_e\|^2}\biggr). 
\end{align*}
Thus,
\begin{align}\label{ueformula}
\xi^{-1}(B_e) &= \tau^{-1}(P_3(2^{-1/2}gB_e)) \notag \\
&= \frac{1}{1+\frac{1}{8}g^2\|B_e\|^2}
\begin{pmatrix}
1-\frac{1}{8}g^2\|B_e\|^2 + \frac{1}{\sqrt{2}}\I g B_e^1 & \frac{1}{\sqrt{2}} g B_e^2 + \I \frac{1}{\sqrt{2}}g B_e^3\\
-\frac{1}{\sqrt{2}}gB_e^2 + \I \frac{1}{\sqrt{2}}gB_e^3 & 1-\frac{1}{8}g^2\|B_e\|^2 - \I \frac{1}{\sqrt{2}}gB_e^1
\end{pmatrix}.
\end{align}
Let $e_1, e_2, e_3, e_4$ be four edges bounding a plaquette $p$, in the order described prior to the statement of the lemma. Let $U_{e_i}:= \xi^{-1}(B_{e_i})$ for $1\le i\le 4$, and let $U_p:=U_{e_1}U_{e_2}U_{e_3}^{-1}U_{e_4}^{-1}$. Note that by the above formula, we have that for each $i$, 
\begin{align}\label{ueformula2}
U_{e_i} &= (1+O(g^2 B_{e_i}^2)) \biggl(I +\frac{1}{\sqrt{2}} g 
\begin{pmatrix}
\I B_{e_i}^1 & B_{e_i}^2 + \I B_{e_i}^3\\
-B_{e_i}^2 + \I B_{e_i}^3 & -\I B_{e_i}^1
\end{pmatrix}
+ O(g^2B_{e_i}^2)\biggr)\notag \\
&= I + \frac{1}{\sqrt{2}}g 
\begin{pmatrix}
\I B_{e_i}^1 & B_{e_i}^2 + \I B_{e_i}^3\\
-B_{e_i}^2 + \I B_{e_i}^3 & -\I B_{e_i}^1
\end{pmatrix} 
+ O(g^2 B_{e_i}^4).
\end{align}
Now, by \eqref{ueformula} and the fact that $U_{e_i}\in \sutwo$, we have
\begin{align*}
U_{e_i}^{-1} &= \frac{1}{1+\frac{1}{8}g^2\|B_{e_i}\|^2}
\begin{pmatrix}
1-\frac{1}{8}g^2\|B_{e_i}\|^2 - \I \frac{1}{\sqrt{2}}g B_{e_i}^1 & -\frac{1}{\sqrt{2}}g B_{e_i}^2 - \I \frac{1}{\sqrt{2}}g B_{e_i}^3\\
\frac{1}{\sqrt{2}}gB_{e_i}^2 - \I \frac{1}{\sqrt{2}}gB_{e_i}^3 & 1-\frac{1}{8}g^2\|B_{e_i}\|^2 + \I \frac{1}{\sqrt{2}}gB_{e_i}^1
\end{pmatrix}.
\end{align*}
Thus, 
\begin{align*}
U_{e_i}^{-1} 
&= I + \frac{1}{\sqrt{2}} g 
\begin{pmatrix}
-\I B_{e_i}^1 & -B_{e_i}^2 - \I B_{e_i}^3\\
B_{e_i}^2 - \I B_{e_i}^3 & \I B_{e_i}^1
\end{pmatrix} 
+ O(g^2 B_{e_i}^4).
\end{align*}
From this, it follows that
\begin{align*}
U_p &= I + \frac{1}{\sqrt{2}} g 
\begin{pmatrix}
\I B_p^1 & B_p^2 + \I B_p^3\\
-B_p^2 + \I B_p^3 & - \I B_p^1
\end{pmatrix} 
+ O(g^2B_p^{16}).
\end{align*}
Together with \eqref{newdensity} and \eqref{ueformula2}, this completes the proof.
\end{proof}

\subsection{Proof of Theorem \ref{sutwothm}}\label{sutwoproof}
Throughout this subsection, $C, C_1,C_2,\ldots$ will denote positive constants whose values depend only on $d$, and may change from line to line. Consider $\sutwo$ theory defined on the finite box $\Lambda = \{-L,\ldots,L\}^d$. Take some $M< L$ and let $\Lambda' := \{-M,\ldots,M\}^d$. Let $E'$ denote the set of oriented nearest-neighbor edges of $\Lambda'$ ({\it not} identifying opposite faces), and let $\partial E'$ denote the boundary edges.  Fix some $\delta_0, \delta \in (g,1)$, with $\delta_0<\delta$, to be chosen later. Let $\me$ be the event that $\|I-V_e\|\le \delta$ for all $e\in E'\setminus \partial E'$ and $\|I-V_e\|\le \delta_0$ for all $e\in \partial E'$. Let $\nu$ denote the law of $V$ and $\nu'$ denote the law of $V$ conditional on the event $\me$. The following lemma shows that $\nu$ and $\nu'$ are close in total variation distance under suitable choices of $M$ and $\delta$ (depending on $\alpha$ and $g$). 
\begin{lmm}\label{tvlmm1}
The total variation distance between $\nu$ and $\nu'$ is bounded above by
\[
\frac{CM^d}{\alpha^4 g^2\delta^2} + \frac{CM^d\log\alpha}{\alpha^2\delta^2}+ \frac{CM^{d-1}}{\alpha^4 g^2\delta_0^2} + \frac{CM^{d-1}\log\alpha}{\alpha^2\delta_0^2},
\]
where $C$ depends only on $d$.
\end{lmm}
\begin{proof}
By Lemma \ref{keylmm},
\begin{align}\label{numebound}
\nu(\me) &\ge 1 - \sum_{e\in E'\setminus\partial E'} \P(\|I-V_e\|>\delta)  - \sum_{e\in \partial E'} \P(\|I-V_e\|> \delta_0)\notag \\
&\ge 1 - \frac{1}{\delta^2} \sum_{e\in E'}\E\|I-V_e\|^2 - \frac{1}{\delta_0^2} \sum_{e\in \partial E'}\E\|I-V_e\|^2\notag  \\
&\ge 1 - \frac{CM^d}{\alpha^4 g^2\delta^2} - \frac{CM^d\log\alpha}{\alpha^2\delta^2} - \frac{CM^{d-1}}{\alpha^4 g^2\delta_0^2} - \frac{CM^{d-1}\log\alpha}{\alpha^2\delta_0^2}. 
\end{align}
Thus, for any event $\ma$, 
\begin{align*}
|\nu(\ma) - \nu'(\ma)| &= \biggl|\nu(\ma) - \frac{\nu(\ma \cap \me)}{\nu(\me)}\biggr|\\
&= \frac{|\nu(\ma)\nu(\me)-\nu(\ma \cap \me)|}{\nu(\me)}\\
&\le \nu(\ma)\frac{1-\nu(\me)}{\nu(\me)} + \frac{\nu(\ma \cap \me^c)}{\nu(\me)} \\
&\le \nu(\ma)\frac{\nu(\me^c)}{\nu(\me)} + \frac{\nu(\me^c)}{\nu(\me)} \le \frac{2\nu(\me^c)}{\nu(\me)},
\end{align*}
where $\me^c$ denotes the complement of the event $\me$. By this bound and the inequality~\eqref{numebound}, we get the desired upper bound on $|\nu(\ma)-\nu'(\ma)|$.
\end{proof}

The next lemma gives an equivalent way of expressing the event $\me$ in terms of the field $A$.
\begin{lmm}\label{aelmm}
The event $\me$ happens if and only if for all $e\in E'\setminus\partial E'$
\begin{align*}
\|A_e\| \le \frac{2^{3/2} \delta}{g \sqrt{8-\delta^2}},
\end{align*}
and for all $e\in \partial E'$, 
\begin{align*}
\|A_e\| \le \frac{2^{3/2} \delta_0}{g \sqrt{8-\delta_0^2}}.
\end{align*}
\end{lmm}
\begin{proof}
Note that by \eqref{iunorm}, $\|I-V_e\|\le \delta$ if and only if $\Re(\tr(V_e)) \ge 2-\frac{1}{2}\delta^2$. This is equivalent to saying that the first coordinate of $\tau(V_e)$ is at least $1-\frac{1}{4}\delta^2$. But by \eqref{pnct}, this happens if and only if 
\[
\|P_3^{-1}(\tau(V_e))\|^2 \le \frac{4\delta^2}{8-\delta^2}. 
\]
Since $A_e = \sqrt{2}g^{-1}P_3^{-1}(\tau(V_e))$, this completes the proof.
\end{proof}
Take any $D = (D_e)_{e\in \partial E'} \in (\R^3)^{\partial E'}$. Let $\gamma_D$ denote the law of $A' := (A_e)_{e\in E'}$ given that $\me$ holds and $A_e = D_e$ for all $e\in \partial E'$. Let $\zeta_D$ denote the joint law of three independent discrete Proca fields on $\Lambda'$ with boundary condition $D$ and parameter $\frac{1}{2}\alpha^2 g^2$. The following lemma shows that $\gamma_D$ and $\zeta_D$ are close in total variation distance under suitable conditions.
\begin{lmm}\label{convlmm}
For any boundary  condition $D$ satisfying the constraint imposed by $\me$, the total variation distance between $\gamma_D$ and $\zeta_D$ is bounded above by 
\[
C_1(g^2 + \alpha^2 g^4)(g^{-1}\delta)^{16}M^d +  C_1 M^d e^{-C_2\alpha^2\delta^2},
\]
provided that 
\begin{align}\label{mcond}
M^{d-1} (\alpha g)^{(d-6)/2}\delta_0\delta^{-1} \le C_3,
\end{align}
where $C_1$, $C_2$ and $C_3$ are positive constants depending only on $d$. 
\end{lmm}
\begin{proof}
By Lemma \ref{approxdens}, the probability density of $\gamma_D$ at a point $B\in (\R^3)^{E'}$ satisfying the given boundary values is proportional to
\begin{align}\label{pexpression}
p(B) &:= \exp\biggl(- \frac{1}{2}\sum_{p\in P'}(\|B_p\|^2 + O(g^2 B^{16}_p)) \notag \\
&\hskip1in - \frac{\alpha^2g^2}{4}\sum_{e\in E'}(\|B_e\|^2 + O(g^2 B^4_e))\biggr)1_{\me}(B),
\end{align}
where $1_\me(B) = 1$ if $B$ satisfies the condition imposed by $\me$ (as in Lemma \ref{aelmm}), and $0$ if not. 
On the other hand, the probability density of $\zeta_D$ is proportional to
\begin{align}\label{qexpression}
q(B) := \exp\biggl(-\frac{1}{2} \sum_{p\in P'}\|B_p\|^2 - \frac{\alpha^2g^2}{4}\sum_{e\in E'}\|B_e\|^2\biggr).
\end{align}
Let $Z_p$ and $Z_q$ be the normalizing constants of $p$ and $q$. Then by Lemma \ref{aelmm} and the identities \eqref{pexpression} and \eqref{qexpression} (and the fact that $g^{-1}\delta \ge 1$), we get
\begin{align}\label{zpzq}
Z_p &= \int p(B) dB \notag \\
&\ge e^{O((g^2 + \alpha^2 g^4)(g^{-1}\delta)^{16}M^d) }\int q(B) 1_{\me}(B) dB\notag \\
&= \rho e^{O((g^2 + \alpha^2 g^4)(g^{-1}\delta)^{16}M^d) }Z_q,
\end{align}
where
\[
\rho := \frac{\int q(B) 1_{\me}(B) dB}{\int q(B) dB}. 
\]
Now, if $1_\me(B)=0$, then $p(B)=0$, and hence, $(Z_p^{-1}p(B) - Z_q^{-1}q(B))^+ = 0$ (where $x^+$ denotes the positive part of a real number $x$). On the other hand, if $1_\me(B)=1$, then by~\eqref{zpzq},  \eqref{pexpression} and~\eqref{qexpression}, we get
\begin{align*}
Z_p^{-1}p(B) &\le Z_q^{-1} \rho^{-1} e^{O((g^2 + \alpha^2 g^4)(g^{-1}\delta)^{16}M^d) } q(B).
\end{align*}
Thus, if $1_\me(B)=1$, then we have
\begin{align*}
(Z_p^{-1} p(B) - Z_q^{-1}q(B))^+ &= \biggl(\frac{Z_p^{-1}p(B)}{Z_q^{-1} q(B)} - 1\biggr)^+ Z_q^{-1} q(B)\\
&\le (\rho^{-1}O((g^2 + \alpha^2 g^4)(g^{-1}\delta)^{16}M^d) + O(\rho^{-1} -1)) Z_q^{-1} q(B),
\end{align*}
provided that 
\begin{align}\label{gcondition}
(g^2 + \alpha^2 g^4)(g^{-1}\delta)^{16}M^d \le 1.
\end{align} 
Thus, under \eqref{gcondition}, 
\begin{align}\label{tvfinal}
&\int (Z_p^{-1}p(B)-Z_q^{-1}q(B))^+ dB = \int (Z_p^{-1}p(B)-Z_q^{-1}q(B))^+1_{\me}(B) dB\notag \\
&\le  (\rho^{-1}O((g^2 + \alpha^2 g^4)(g^{-1}\delta)^{16}M^d) + O(\rho^{-1} - 1)) \int Z_q^{-1} q(B) dB\notag \\
&=  \rho^{-1} O((g^2 + \alpha^2 g^4)(g^{-1}\delta)^{16}M^d) +  O(\rho^{-1} - 1),
\end{align}
where the integrals are over the set of $B$ satisfying the given boundary values. 

Now, let $W = (W^1, W^2, W^3)$ be a random configuration drawn from the probability measure $\zeta_D$. Take any $1\le i\le 3$. Define $\ve := \frac{1}{\sqrt{2}} \alpha g$, so that $W^i$ is a discrete Proca field on $\Lambda'$ with boundary condition $D$ and parameter $\ve^2$. Note that by Lemma \ref{aelmm}, 
\begin{align*}
\|D\|^2 &\le CM^{d-1} g^{-2}\delta_0^2.  
\end{align*}
Let $q$, $R$ and $S$ be as in the proof of Lemma \ref{normallmm}, but with $E$ replaced by $E'$. Then by the same calculation as for \eqref{qbbound}, we get that for an edge $e\in E'$ that is at an $\ell^\infty$ distance $r$ from the boundary $\partial E'$, 
\begin{align*}
\|q(e)\|^2 &\le C_1M^{d-1} \ve^{-2d} e^{-C_2 \ve^2 r}.
\end{align*}
Just as we had \eqref{expvar}, we now have for all $e\in E'$, 
\begin{align*}
&\E(W^i(e)) = \ve^{(d-2)/2} q(e)^* S^{-1} D, \\
&\var(W^i(e)) = \ve^{d-2}(R(e,e) - q(e)^*S^{-1}q(e))\le \ve^{d-2}R(e,e).
\end{align*}
As in Lemma \ref{normallmm}, we deduce that $\|S^{-1}\|= O(\ve^{d-2})$. Thus, for any $e\in E'$,
\begin{align*}
|\E(W^i(e))| &\le \ve^{(d-2)/2}\|q(e)\| \|S^{-1}\| \|D\| \\
&= O( M^{d-1} \ve^{(d-6)/2} g^{-1}\delta_0). 
\end{align*}
Also, as in the proof of Theorem \ref{convthm}, we have that $\|R\| = O(\ve^{-d})$. Thus, $\var(W^i(e)) = O(\ve^{-2})$ for each $e$. Since $W^i(e)$ is a Gaussian random variable for each $e$ and $i$, this shows that (by Lemma \ref{aelmm})
\begin{align*}
1-\rho &\le \sum_{i=1}^3 \sum_{e\in E'\setminus\partial E'} \P\biggl(|W^i(e)| >  \frac{2^{3/2} \delta}{g\sqrt{3} \sqrt{8-\delta^2}}\biggr)\\
&\le C_1 M^d \exp\biggl(-\frac{C_2(g^{-1}\delta)^2}{\ve^{-2}}\biggr),
\end{align*}
provided that $g^{-1}\delta \ge C M^{d-1} \ve^{(d-6)/2} g^{-1}\delta_0$ 
for a sufficiently large constant $C$ (depending only on $d$). 
Substituting the value of $\ve$ and using this bound in \eqref{tvfinal} completes the proof.
\end{proof}

We are now ready to complete the proof of Theorem \ref{sutwothm}. By Lemma \ref{scalinglmm}, it suffices to prove the claim for $c=\sqrt{2}$, that is, $\alpha g =\sqrt{2}\ve$. Suppose that $g$ and  $L$ vary as in the statement of the theorem. Let $\kappa>0$ be the number such that $g= (\sqrt{2}\ve)^{1/\kappa}$. 

By definition of $\kappa$, we have $\alpha = (\sqrt{2}\ve)/g = g^{\kappa-1}$. Let $\delta := g^{1-a\kappa}$, $\delta_0 := g^{1-b\kappa}$ and $M:= \lfloor \sqrt{2}g^{-4\kappa}\rfloor$ (assuming $M<L$, which we will eventually guarantee below), where $a$ and $b$ will be chosen later. Let $\tilde{A}\in (\R^3)^E$ be a random configuration drawn from the law of $A$ conditional on the event $\me$. Let $\tilde{Y}$ be defined using $\tilde{A}$ in place of $A$ in \eqref{yfroma}, and define $\tilde{Z}$ using $\tilde{Y}$ just as we defined $Z$ using $Y$ in~\eqref{zfromy}. Note that with the above choices of $\delta$ and $M$, we have
\begin{align*}
\frac{M^d}{\alpha^4 g^2\delta^2} + \frac{M^d\log\alpha}{\alpha^2\delta^2} &\le \frac{2^{d/2}g^{-4\kappa d}}{g^{4\kappa-4}g^2 g^{2-2a\kappa }} + \frac{2^{d/2}g^{-4\kappa d} \log (1/g)}{g^{2\kappa-2} g^{2-2a\kappa}} \\
&= 2^{d/2}g^{\kappa(-4d - 4 + 2a) } + 2^{d/2}g^{\kappa(-4d -2 + 2a)}\log(1/g).
\end{align*}
This goes to zero as $g\to 0$ if 
\begin{align}\label{abcond1}
a > 2d+2.
\end{align}
Next, note that 
\begin{align*}
\frac{M^{d-1}}{\alpha^4 g^2\delta_0^2} + \frac{M^{d-1}\log\alpha}{\alpha^2\delta_0^2} &\le  \frac{2^{(d-1)/2}g^{-4\kappa (d-1)}}{g^{4\kappa-4}g^2 g^{2-2b\kappa }} + \frac{2^{(d-1)/2}g^{-4\kappa (d-1)} \log (1/g)}{g^{2\kappa-2} g^{2-2b\kappa}} \\
&= 2^{(d-1)/2}g^{\kappa(-4(d-1) - 4 + 2b) } + 2^{(d-1)/2}g^{\kappa(-4(d-1) -2 + 2b)}\log(1/g).
\end{align*}
This goes to zero as $g\to 0$ if 
\begin{align}\label{abcond2}
b > 2(d-1)+2. 
\end{align}
Thus, if both \eqref{abcond1} and \eqref{abcond2} hold, then Lemma \ref{tvlmm1} implies that to prove Theorem~\ref{sutwothm}, it suffices to prove convergence for $\tilde{Z}$ instead of $Z$.

Let $W$ be as in the proof of Lemma \ref{convlmm}, with the set of boundary values given by $D_e = \tilde{A}_e$ for $e\in \partial E'$. Define a triple of random $1$-forms $Q = (Q^1, Q^2, Q^3)$ using $W$, via the same procedure that we used to define $Z$ using $A$ in \eqref{yfroma} and \eqref{zfromy}.  Note that
\begin{align*}
M^{d-1} (\alpha g)^{(d-6)/2} \delta_0\delta^{-1}&\le 2^{(d-1)/2}g^{\kappa(-4(d-1) + (d-6)/2 + a-b)} = 2^{(d-1)/2} g^{\kappa ( (2-7d)/2 + a-b )}. 
\end{align*}
To make this go to zero as $g\to 0$, we need
\begin{align}\label{abcond4}
a-b > \frac{7d - 2}{2}. 
\end{align}
Thus, if  \eqref{abcond4} holds, then the condition \eqref{mcond} is satisfied for sufficiently small $g$. Now note that with $C_2$ being the constant from Lemma \ref{convlmm},
\begin{align*}
&(g^2 + \alpha^2 g^4)(g^{-1}\delta)^{16}M^d +   M^d e^{-C_2\alpha^2\delta^2}\\
&\le 2^{d/2}g^2 g^{-16 a\kappa}g^{-4\kappa d} + 2^{d/2}g^{-4 \kappa d} e^{-C_2 g^{2\kappa-2} g^{2-2a\kappa}}\\
&= 2^{d/2}g^{2-\kappa(16a + 4d)} + 2^{d/2}g^{-4\kappa d} e^{-C_2g^{(2-2a)\kappa}}.
\end{align*}
To make this go to zero as $g\to 0$, it suffices to have
\begin{align}\label{abcond5}
a > 1, \ \ \ \kappa < \frac{2}{16a + 4d}. 
\end{align}
Thus, by Lemma \ref{convlmm}, if the conditions \eqref{abcond4} and \eqref{abcond5} hold, then to prove that $\tilde{Z}$ converges to the required limit, it suffices to prove the same for $Q$. First of all, since none of the bounds depend on $L$, we can fix all other parameters and send $L$ to infinity. Next, let us send $g\to 0$, keeping $\kappa$ fixed, so that $\ve \to 0$, $M\to \infty$ and $\alpha \to \infty$. Since $M\ve^3\ge g^{-\kappa} \to\infty$ as $\ve \to 0$ and $\|D\|$ is growing at most polynomially in $M$, Lemma \ref{normallmm} implies  that for any $f\in \ma(\R^d)$, $Q(f)$ converges in law to the required limit. Thus, to complete the proof of Theorem \ref{sutwothm}, we only need to find $a$ and $b$ such that the conditions \eqref{abcond1}--\eqref{abcond5} are satisfied. This is accomplished by choosing, for example, $b= 2d+1$ and $a = 6d$, and noting that $\kappa< 1/50d$. Since we are considering $g = O(\ve^{50d})$, $\kappa$ is eventually less than $1/50d$ as $\ve \to 0$.

\subsection{The key estimate for $\uone$ theory}
Let $V$ be the field obtained from $\uone$ theory after unitary gauge fixing in Subsection~\ref{uonesec}, on the finite box $\Lambda$. Note that for any $U\in \uone$, 
\begin{align}\label{iunorm2}
|1-U|^2 &=  2 - 2\Re(U). 
\end{align}
Together with Lemma \ref{vdens1}, this shows that  probability density of $V$ is proportional to $e^{-H(U)}$, where 
\[
H(U) := \frac{1}{2g^2}\sum_{p\in P} |1-U_p|^2 +\frac{\alpha^2}{2} \sum_{e\in E} |1-U_e|^2.
\]
The following lemma gives the key estimate for proving Theorem \ref{uonethm}.
\begin{lmm}\label{keylmm2}
Suppose that $\alpha \ge 2$ and $\alpha g\le 1$. Then for any $e\in E$, 
\[
\E|1 - V_e|^2 \le \frac{C}{\alpha^4 g^2} + \frac{C\log \alpha}{\alpha^2},
\]
where $C$ depends only on the dimension $d$. 
\end{lmm}
\begin{proof}
Let $\Sigma'$ be the subset of $\Sigma$ consisting of all $U$ such that $|1-U_e |\le \alpha^{-1}$ for all $e$.  It is easy to see that the normalized Haar measure of $\Sigma'$ is $\ge (C_1\alpha)^{-C_2L^d}$ (where $C_1$ and $C_2$ depend only on $d$), because the Haar measure on $\uone$ is just the uniform probability measure on the unit circle $S^1$. The rest of the proof is exactly like the proof of Lemma~\ref{keylmm}.
\end{proof}

\subsection{Local free field approximation for $\uone$ theory}
Adopting the same notational conventions as in Lemma \ref{approxdens}, we have the following result.
\begin{lmm}\label{approxdens2}
The probability density function of the random field $A$, at a point $B\in \R^E$, is proportional to
\begin{align*}
\exp\biggl(-\frac{1}{2} \sum_{p\in P}(B_p^2 + O(g^2 B^{16}_p))- \frac{\alpha^2g^2}{2}\sum_{e\in E}(B_e^2 + O(g^2 B^4_e))\biggr). 
\end{align*}
\end{lmm}
\begin{proof}
Let $P_1:\R \to S^1$ be the map defined in Subsection \ref{stereosec}. Let $\xi := g^{-1} P_1^{-1}$, so that $A_e = \xi(V_e)$. By Lemma \ref{vdens1}, Corollary \ref{densitycor}, and the identity \eqref{iunorm2}, the probability density of $A$ with respect to Lebesgue measure on $\R^E$ at a point $B\in \R^E$ is proportional to
\begin{align}\label{newdensity2}
\exp\biggl(-\frac{1}{2g^2}\sum_{p\in P} |1 - \xi^{-1}(B)_p|^2 - \frac{\alpha^2}{2} \sum_{e\in E} |1-\xi^{-1}(B)_e|^2\biggr)\prod_{e\in E} \frac{1}{4 + g^2B_e^2}. 
\end{align}
Now, note that
\begin{align*}
\xi^{-1}(B_e) = P_1(g B_e) &= \frac{4 - g^2 B_e^2}{4+g^2B_e^2}+ \I \frac{4gB_e}{4+g^2B_e^2}. 
\end{align*}
Let $e_1, e_2, e_3, e_4$ be four edges bounding a plaquette $p$, in the order described prior to the statement of Lemma \ref{approxdens}. Let $U_{e_i}:= \xi^{-1}(B_{e_i})$ for $1\le i\le 4$, and let $U_p:=U_{e_1}U_{e_2}U_{e_3}^{-1}U_{e_4}^{-1}$. Note that by the above formula, we have that for each $i$, 
\begin{align}\label{ueformula22}
U_{e_i} &=  1 + \I g B_{e_i} + O(g^2B_{e_i}^4).
\end{align}
Since $U_{e_i}\in \uone$, this also shows that
\begin{align*}
U_{e_i}^{-1} &= 1 - \I g B_{e_i} + O(g^2B_{e_i}^4).
\end{align*}
Thus,
\begin{align*}
U_p &= 1 + \I gB_p + O(g^2B_p^{16}).
\end{align*}
Together with \eqref{newdensity2} and \eqref{ueformula22}, this completes the proof.
\end{proof}

\subsection{Proof of Theorem \ref{uonethm}}\label{uoneproof}
Throughout this subsection, $C, C_1,C_2,\ldots$ will denote positive constants whose values depend only on $d$, and may change from line to line. Take some $M< L$ and let $\Lambda' := \{-M,\ldots,M\}^d$. Let $E'$ denote the set of oriented nearest-neighbor edges of $\Lambda'$ ({\it not} identifying opposite faces), and let $\partial E'$ denote the boundary edges.  Fix some $\delta_0, \delta \in (g,1)$, with $\delta_0<\delta$, to be chosen later. Let $\me$ be the event that $|1-V_e|\le \delta$ for all $e\in E'\setminus \partial E'$ and $|1-V_e|\le \delta_0$ for all $e\in \partial E'$. Let $\nu$ denote the law of $V$ and $\nu'$ denote the law of $V$ conditional on the event $\me$. The following lemma shows that $\nu$ and $\nu'$ are close in total variation distance under suitable choices of $M$ and $\delta$ (depending on $\alpha$ and~$g$). 
\begin{lmm}\label{tvlmm12}
The total variation distance between $\nu$ and $\nu'$ is bounded above by
\[
\frac{CM^d}{\alpha^4 g^2\delta^2} + \frac{CM^d\log\alpha}{\alpha^2\delta^2}+ \frac{CM^{d-1}}{\alpha^4 g^2\delta_0^2} + \frac{CM^{d-1}\log\alpha}{\alpha^2\delta_0^2},
\]
where $C$ depends only on $d$.
\end{lmm}
\begin{proof}
This is exactly like the proof of Lemma \ref{tvlmm1}, but using Lemma \ref{keylmm2} instead of Lemma~\ref{keylmm}. We omit the details.
\end{proof}

The next lemma gives an equivalent way of expressing the event $\me$ in terms of the field $A$.
\begin{lmm}\label{aelmm2}
The event $\me$ happens if and only if for all $e\in E'\setminus\partial E'$
\begin{align*}
|A_e| \le \frac{2 \delta}{g \sqrt{4-\delta^2}},
\end{align*}
and for all $e\in \partial E'$, 
\begin{align*}
|A_e| \le \frac{2 \delta_0}{g \sqrt{4-\delta_0^2}},
\end{align*}
\end{lmm}
\begin{proof}
Note that by \eqref{iunorm2}, $|1-V_e|\le \delta$ if and only if $\Re(V_e) \ge 1-\frac{1}{2}\delta^2$.  By \eqref{pnct}, this happens if and only if 
\[
|P_1^{-1}(V_e)|^2 \le \frac{4\delta^2}{4-\delta^2}. 
\]
Since $A_e = g^{-1}P_1^{-1}(V_e)$, this completes the proof.
\end{proof}
Take any $D = (D_e)_{e\in \partial E'} \in \R^{\partial E'}$. Let $\gamma_D$ denote the law of $A' := (A_e)_{e\in E'}$ given that $\me$ holds and $A_e = D_e$ for all $e\in \partial E'$. Let $\zeta_D$ denote the law of the discrete Proca field on $\Lambda'$ with boundary condition $D$ and parameter $\alpha^2 g^2$. The following lemma shows that $\gamma_D$ and $\zeta_D$ are close in total variation distance under suitable conditions.
\begin{lmm}\label{convlmm2}
For any boundary  condition $D$ satisfying the constraint imposed by $\me$, the total variation distance between $\gamma_D$ and $\zeta_D$ is bounded above by 
\[
C_1(g^2 + \alpha^2 g^4)(g^{-1}\delta)^{16}M^d +  C_1 M^d e^{-C_2\alpha^2\delta^2},
\]
provided that 
\begin{align*}
M^{d-1} (\alpha g)^{(d-6)/2}\delta_0\delta^{-1} \le C_3,
\end{align*}
where $C_1$, $C_2$ and $C_3$ are positive constants depending only on $d$. 
\end{lmm}
\begin{proof}
This proof is an exact copy of the proof of Lemma \ref{convlmm}, with minor modifications at the appropriate locations.
\end{proof}

Given Lemma \ref{convlmm2}, the rest of the proof of Theorem \ref{uonethm} goes through exactly as the proof of Theorem \ref{sutwothm}, with essentially no changes.

\section*{Acknowledgements}
I thank David Brydges, Erhard Seiler, Steve Shenker, Tom Spencer and Edward Witten for many helpful conversations, and the two anonymous referees for a number of excellent suggestions. This work was partially supported by NSF grants DMS-2113242 and DMS-2153654, and a membership at the Institute for Advanced Study.

\bibliographystyle{abbrvnat}

\bibliography{myrefs}

\begin{thebibliography}{43}
\providecommand{\natexlab}[1]{#1}
\providecommand{\url}[1]{\texttt{#1}}
\expandafter\ifx\csname urlstyle\endcsname\relax
  \providecommand{\doi}[1]{doi: #1}\else
  \providecommand{\doi}{doi: \begingroup \urlstyle{rm}\Url}\fi

\bibitem[Aru et~al.(2022)Aru, Garban, and Sep{\'u}lveda]{aruetal22}
J.~Aru, C.~Garban, and A.~Sep{\'u}lveda.
\newblock {Percolation for 2D classical Heisenberg model and exit sets of
  vector valued GFF}.
\newblock \emph{arXiv preprint arXiv:2212.06767}, 2022.

\bibitem[Balaban(1982{\natexlab{a}})]{balaban82a}
T.~Balaban.
\newblock {$($Higgs$)_{2, 3}$ quantum fields in a finite volume. I.}
\newblock \emph{Communications in Mathematical Physics}, 85\penalty0
  (4):\penalty0 603--626, 1982{\natexlab{a}}.

\bibitem[Balaban(1982{\natexlab{b}})]{balaban82b}
T.~Balaban.
\newblock {$($Higgs$)_{2, 3}$ quantum fields in a finite volume. II.}
\newblock \emph{Communications in Mathematical Physics}, 86\penalty0
  (4):\penalty0 555--594, 1982{\natexlab{b}}.

\bibitem[Balaban(1983{\natexlab{a}})]{balaban83a}
T.~Balaban.
\newblock {$($Higgs$)_{2, 3}$ quantum fields in a finite volume. III.}
\newblock \emph{Communications in Mathematical Physics}, 88\penalty0
  (3):\penalty0 411--445, 1983{\natexlab{a}}.

\bibitem[Balaban(1983{\natexlab{b}})]{balaban83b}
T.~Balaban.
\newblock {Regularity and decay of lattice Green's functions}.
\newblock \emph{Communications in Mathematical Physics}, 89\penalty0
  (4):\penalty0 571--597, 1983{\natexlab{b}}.

\bibitem[Balaban(1985)]{balaban85}
T.~Balaban.
\newblock {Ultraviolet stability of three-dimensional lattice pure gauge field
  theories}.
\newblock \emph{Communications in Mathematical Physics}, 102\penalty0
  (2):\penalty0 255--275, 1985.

\bibitem[Balaban(1989)]{balaban89}
T.~Balaban.
\newblock {Large field renormalization. II. Localization, exponentiation, and
  bounds for the R operation}.
\newblock \emph{Communications in Mathematical Physics}, 122\penalty0
  (3):\penalty0 355--392, 1989.

\bibitem[Balaban et~al.(1984)Balaban, Imbrie, Jaffe, and
  Brydges]{balabanetal84}
T.~Balaban, J.~Imbrie, A.~Jaffe, and D.~Brydges.
\newblock {The mass gap for Higgs models on a unit lattice}.
\newblock \emph{Annals of Physics}, 158\penalty0 (2):\penalty0 281--319, 1984.

\bibitem[Banks and Rabinovici(1979)]{banksrabinovici79}
T.~Banks and E.~Rabinovici.
\newblock {Finite Temperature Behavior of the Lattice Abelian Higgs Model}.
\newblock \emph{Nuclear Physics B}, 160:\penalty0 349--379, 1979.

\bibitem[Borgs and Nill(1987)]{borgsnill87}
C.~Borgs and F.~Nill.
\newblock {The phase diagram of the Abelian lattice Higgs model. A review of
  rigorous results}.
\newblock \emph{Journal of Statistical Physics}, 47:\penalty0 877--904, 1987.

\bibitem[Bringmann and Cao(2024)]{bringmanncao24}
B.~Bringmann and S.~Cao.
\newblock {Global well-posedness of the stochastic Abelian-Higgs equations in
  two dimensions}.
\newblock \emph{arXiv preprint arXiv:2403.16878}, 2024.

\bibitem[Brydges et~al.(1979)Brydges, Fr{\"o}hlich, and Seiler]{brydgesetal79}
D.~Brydges, J.~Fr{\"o}hlich, and E.~Seiler.
\newblock {On the construction of quantized gauge fields. I. General results}.
\newblock \emph{Annals of Physics}, 121\penalty0 (1-2):\penalty0 227--284,
  1979.

\bibitem[Cao and Chatterjee(2021)]{caochatterjee21}
S.~Cao and S.~Chatterjee.
\newblock {A state space for 3D Euclidean Yang-Mills theories}.
\newblock \emph{arXiv preprint arXiv:2111.12813}, 2021.

\bibitem[Cao and Chatterjee(2023)]{caochatterjee23}
S.~Cao and S.~Chatterjee.
\newblock {The Yang-Mills heat flow with random distributional initial data}.
\newblock \emph{Communications in Partial Differential Equations}, 48\penalty0
  (2):\penalty0 209--251, 2023.

\bibitem[Chandra et~al.(2022{\natexlab{a}})Chandra, Chevyrev, Hairer, and
  Shen]{chandraetal22}
A.~Chandra, I.~Chevyrev, M.~Hairer, and H.~Shen.
\newblock {Langevin dynamic for the 2D Yang--Mills measure}.
\newblock \emph{Publications Math{\'e}matiques de l'IH{\'E}S}, 136\penalty0
  (1):\penalty0 1--147, 2022{\natexlab{a}}.

\bibitem[Chandra et~al.(2022{\natexlab{b}})Chandra, Chevyrev, Hairer, and
  Shen]{chandraetal22b}
A.~Chandra, I.~Chevyrev, M.~Hairer, and H.~Shen.
\newblock {Stochastic quantisation of Yang-Mills-Higgs in 3D}.
\newblock \emph{arXiv preprint arXiv:2201.03487}, 2022{\natexlab{b}}.

\bibitem[Chevyrev(2019)]{chevyrev19}
I.~Chevyrev.
\newblock {Yang--Mills measure on the two-dimensional torus as a random
  distribution}.
\newblock \emph{Communications in Mathematical Physics}, 372\penalty0
  (3):\penalty0 1027--1058, 2019.

\bibitem[Chevyrev(2022)]{chevyrev22}
I.~Chevyrev.
\newblock {Stochastic quantization of Yang--Mills}.
\newblock \emph{Journal of Mathematical Physics}, 63\penalty0 (9), 2022.

\bibitem[Chevyrev and Shen(2023)]{chevyrevshen23}
I.~Chevyrev and H.~Shen.
\newblock {Invariant measure and universality of the 2D Yang--Mills Langevin
  dynamic}.
\newblock \emph{arXiv preprint arXiv:2302.12160}, 2023.

\bibitem[Dimock(2018)]{dimock18}
J.~Dimock.
\newblock {Ultraviolet regularity for QED in $d= 3$}.
\newblock \emph{Journal of Mathematical Physics}, 59\penalty0 (1), 2018.

\bibitem[Dimock(2020)]{dimock20}
J.~Dimock.
\newblock {Multiscale block averaging for QED in $d= 3$}.
\newblock \emph{Journal of Mathematical Physics}, 61\penalty0 (3), 2020.

\bibitem[Driver(1987)]{driver87}
B.~K. Driver.
\newblock {Convergence of the $U(1)_4$ lattice gauge theory to its continuum
  limit}.
\newblock \emph{Communications in Mathematical Physics}, 110\penalty0
  (3):\penalty0 479--501, 1987.

\bibitem[Driver(1989)]{driver89}
B.~K. Driver.
\newblock {$\mathrm{YM}_2$: Continuum expectations, lattice convergence, and
  lassos}.
\newblock \emph{Communications in Mathematical Physics}, 123:\penalty0
  575--616, 1989.

\bibitem[Fine(1991)]{fine91}
D.~S. Fine.
\newblock {Quantum Yang--Mills on a Riemann surface}.
\newblock \emph{Communications in Mathematical Physics}, 140:\penalty0
  321--338, 1991.

\bibitem[Fradkin and Shenker(1979)]{fradkinshenker79}
E.~Fradkin and S.~H. Shenker.
\newblock {Phase diagrams of lattice gauge theories with Higgs fields}.
\newblock \emph{Physical Review D}, 19\penalty0 (12):\penalty0 3682, 1979.

\bibitem[Ginibre and Velo(1975)]{ginibrevelo75}
J.~Ginibre and G.~Velo.
\newblock {The free Euclidean massive vector field in the St{\"u}ckelberg
  gauge}.
\newblock \emph{Annales de l'IHP Physique Th{\'e}orique}, 22\penalty0
  (3):\penalty0 257--264, 1975.

\bibitem[Giuliani and Ott(2023)]{giulianiott23}
A.~Giuliani and S.~Ott.
\newblock Low temperature asymptotic expansion for classical $o(n)$ vector
  models.
\newblock \emph{arXiv preprint arXiv:2302.07299}, 2023.

\bibitem[Glimm and Jaffe(1987)]{glimmjaffe87}
J.~Glimm and A.~Jaffe.
\newblock \emph{Quantum Physics: A Functional Integral Point Of View}.
\newblock Springer-Verlag, New York, second edition, 1987.

\bibitem[Glimm et~al.(1976)Glimm, Jaffe, and Spencer]{glimmetal76}
J.~Glimm, A.~Jaffe, and T.~Spencer.
\newblock {A convergent expansion about mean field theory: I. The expansion}.
\newblock \emph{Annals of Physics}, 101\penalty0 (2):\penalty0 610--630, 1976.

\bibitem[Gross(1974)]{gross74}
L.~Gross.
\newblock {The free Euclidean Proca and electromagnetic fields}.
\newblock In \emph{Report to the Cumberland Lodge Conference on Functional
  Integration and its Applications}, 1974.

\bibitem[Gross(1983)]{gross83}
L.~Gross.
\newblock {Convergence of $U(1)_3$ lattice gauge theory to its continuum
  limit}.
\newblock \emph{Communications in Mathematical Physics}, 92\penalty0
  (2):\penalty0 137--162, 1983.

\bibitem[Gross(2022)]{gross22}
L.~Gross.
\newblock \emph{{The Yang-Mills heat equation with finite action in three
  dimensions}}.
\newblock Memoirs of the American Mathematical Society. American Mathematical
  Society, 2022.

\bibitem[Gross et~al.(1989)Gross, King, and Sengupta]{grossetal89}
L.~Gross, C.~King, and A.~Sengupta.
\newblock {Two dimensional Yang-Mills theory via stochastic differential
  equations}.
\newblock \emph{Annals of Physics}, 194\penalty0 (1):\penalty0 65--112, 1989.

\bibitem[Jaffe and Witten(2006)]{jaffewitten06}
A.~Jaffe and E.~Witten.
\newblock {Quantum Yang--Mills theory}.
\newblock In \emph{{The millennium prize problems}}, pages 129--152. Clay
  Mathematics Institute, Cambridge, MA., 2006.

\bibitem[Kennedy and King(1986)]{kennedyking86}
T.~Kennedy and C.~King.
\newblock {Spontaneous symmetry breakdown in the abelian Higgs model}.
\newblock \emph{Communications in Mathematical Physics}, 104\penalty0
  (2):\penalty0 327--347, 1986.

\bibitem[L{\'e}vy(2003)]{levy03}
T.~L{\'e}vy.
\newblock \emph{{Yang--Mills measure on compact surfaces}}.
\newblock Memoirs of the American Mathematical Society. American Mathematical
  Soc., 2003.

\bibitem[Magnen et~al.(1993)Magnen, Rivasseau, and
  S{\'e}n{\'e}or]{magnenetal93}
J.~Magnen, V.~Rivasseau, and R.~S{\'e}n{\'e}or.
\newblock {Construction of $\mathrm{YM}_4$ with an infrared cutoff}.
\newblock \emph{Communications in Mathematical Physics}, 155:\penalty0
  325--383, 1993.

\bibitem[Osterwalder and Seiler(1978)]{osterwalderseiler78}
K.~Osterwalder and E.~Seiler.
\newblock Gauge field theories on a lattice.
\newblock \emph{Annals of Physics}, 110\penalty0 (2):\penalty0 440--471, 1978.

\bibitem[Seiler(1982)]{seiler82}
E.~Seiler.
\newblock \emph{{Gauge Theories as a Problem of Constructive Quantum Field
  Theory and Statistical Mechanics}}.
\newblock Springer, Berlin, Heidelberg, 1982.

\bibitem[Sengupta(1997)]{sengupta97}
A.~Sengupta.
\newblock \emph{Gauge theory on compact surfaces}.
\newblock Memoirs of the American Mathematical Society. American Mathematical
  Soc., 1997.

\bibitem[Shen(2021)]{shen21}
H.~Shen.
\newblock {Stochastic quantization of an Abelian gauge theory}.
\newblock \emph{Communications in Mathematical Physics}, 384\penalty0
  (3):\penalty0 1445--1512, 2021.

\bibitem[Shen et~al.(2023)Shen, Zhu, and Zhu]{shenetal23}
H.~Shen, R.~Zhu, and X.~Zhu.
\newblock {A stochastic analysis approach to lattice Yang--Mills at strong
  coupling}.
\newblock \emph{Communications in Mathematical Physics}, 400\penalty0
  (2):\penalty0 805--851, 2023.

\bibitem[St{\"u}ckelberg(1938)]{stuckelberg38}
E.~C.~G. St{\"u}ckelberg.
\newblock Interaction energy in electrodynamics and in the field theory of
  nuclear forces.
\newblock \emph{Helv. Phys. Acta}, 11\penalty0 (3):\penalty0 225, 1938.

\end{thebibliography}

\end{document}